\documentclass[12pt,english]{article}
\usepackage[T1]{fontenc}
\usepackage[latin9]{inputenc}
\usepackage{geometry}
\usepackage{float}
\usepackage{mathtools}
\usepackage{amsmath}
\usepackage{amsthm}
\usepackage{amssymb}
\usepackage{microtype}
\usepackage{breakurl}

\makeatletter

\providecommand{\tabularnewline}{\\}

\theoremstyle{plain}
\newtheorem{thm}{\protect\theoremname}[section]
\theoremstyle{plain}
\newtheorem{prop}[thm]{\protect\propositionname}
\theoremstyle{plain}
\newtheorem{lem}[thm]{\protect\lemmaname}
\theoremstyle{remark}
\newtheorem{claim}[thm]{\protect\claimname}
\theoremstyle{plain}
\newtheorem{conjecture}[thm]{\protect\conjecturename}

\usepackage{bm}

\usepackage{tikz}
\usetikzlibrary{arrows}
\usepackage{hyperref}

\usepackage{colortbl}
\usepackage{xparse}
\usepackage{caption}

\usepackage{titlesec}
\titleformat{\section}
  {\normalfont\large\bfseries}{\thesection.}{.8ex}{} 
\titleformat{\subsection}
  {\normalfont\bfseries}{\thesubsection.}{.8ex}{} 

\usetikzlibrary{shapes}

\tikzstyle{pathdefault}=[draw, line width=1, solid, color=black]
\tikzstyle{nodedefault}=[circle, inner sep=1.5, fill=black]
\tikzstyle{empty}=[]
\tikzstyle{nodeellipsis}=[circle, inner sep=0.5, fill=black]
\tikzstyle{pathcolor1}=[draw, line width=1.3, densely dashed, color=red]
\tikzstyle{pathcolor2}=[draw, line width=1.6, densely dotted, color=blue]
\tikzstyle{pathcolorlight}=[draw, line width=1, dotted, color=lightgray]

\tikzstyle{arbpathcolor0}=[line width=1, dashdotted, color=black]
\tikzstyle{arbpathcolor1}=[line width=1, densely dashed, color=red]
\tikzstyle{arbpathdefault}=[line width=1, densely dotted, color=blue]

\newcounter{id}
\newcommand{\drawlinedotswithstyle}[4]{
 \def\x{{#3}}
 \def\y{{#4}}
 \tikzstyle{thispathstyle}=[#1]
 \tikzstyle{thisnodestyle}=[#2]
 \setcounter{id}{-1} 
 \foreach \j in {#3}{\stepcounter{id}} 
 \foreach \i in {1,...,\the\value{id}}{  
  \path[thispathstyle] (\x[\i],\y[\i]) --(\x[\i-1],\y[\i-1]); 
 }
 \foreach \i in {1,...,\the\value{id}}{  
  \node[thisnodestyle] at (\x[\i],\y[\i]) {}; 
 }
 \node[thisnodestyle] at (\x[0],\y[0]) {}; 
}

\DeclareDocumentCommand{\drawlinedots}{ O{pathdefault} O{nodedefault} m m}{\drawlinedotswithstyle{#1}{#2}{#3}{#4}}

\let\originalleft\left
\let\originalright\right
\renewcommand{\left}{\mathopen{}\mathclose\bgroup\originalleft}
\renewcommand{\right}{\aftergroup\egroup\originalright}

\definecolor{davidsonred}{HTML}{AC1A2F} 

\definecolor{green}{RGB}{0, 180, 0}
\definecolor{yellow}{RGB}{180, 180, 0}

\newcommand{\leqnomode}{\tagsleft@true\let\veqno\@@leqno}
\newcommand{\reqnomode}{\tagsleft@false\let\veqno\@@eqno}
\reqnomode

\makeatother

\usepackage{babel}
\providecommand{\claimname}{Claim}
\providecommand{\conjecturename}{Conjecture}
\providecommand{\lemmaname}{Lemma}
\providecommand{\propositionname}{Proposition}
\providecommand{\theoremname}{Theorem}

\begin{document}
\global\long\def\iDes{\operatorname{iDes}}%

\global\long\def\des{\operatorname{des}}%

\global\long\def\maj{\operatorname{maj}}%

\global\long\def\Pk{\operatorname{Pk}}%

\global\long\def\pk{\operatorname{pk}}%

\global\long\def\lpk{\operatorname{lpk}}%

\global\long\def\std{\operatorname{std}}%

\global\long\def\inv{\operatorname{inv}}%

\global\long\def\Des{\operatorname{Des}}%

\global\long\def\Comp{\operatorname{Comp}}%

\global\long\def\occ{\operatorname{occ}}%

\global\long\def\mk{\operatorname{mk}}%

\global\long\def\Av{\operatorname{Av}}%

\global\long\def\ides{\operatorname{ides}}%

\global\long\def\imaj{\operatorname{imaj}}%

\global\long\def\icomaj{\operatorname{icomaj}}%

\global\long\def\comaj{\operatorname{comaj}}%

\global\long\def\ipk{\operatorname{ipk}}%

\global\long\def\ilpk{\operatorname{ilpk}}%

\global\long\def\st{\operatorname{st}}%

\global\long\def\ist{\operatorname{ist}}%

\title{A lifting of the Goulden\textendash Jackson cluster method to the
Malvenuto\textendash Reutenauer algebra}
\author{Yan Zhuang\\
Department of Mathematics and Computer Science\\
Davidson College\texttt{}~\\
\texttt{yazhuang@davidson.edu}}
\maketitle
\begin{abstract}
The Goulden\textendash Jackson cluster method is a powerful tool for
counting words by occurrences of prescribed subwords, and was adapted
by Elizalde and Noy for counting permutations by occurrences of prescribed
consecutive patterns. In this paper, we lift the cluster method for
permutations to the Malvenuto\textendash Reutenauer algebra. Upon
applying standard homomorphisms, our result specializes to both the
cluster method for permutations as well as a $q$-analogue which keeps
track of the inversion number statistic. We construct additional homomorphisms
using the theory of shuffle-compatibility, leading to further specializations
which keep track of various ``inverse statistics'', including the
inverse descent number, inverse peak number, and inverse left peak
number. This approach is then used to derive formulas for counting
permutations by occurrences of two families of consecutive patterns\textemdash monotone
patterns and transpositional patterns\textemdash refined by these
statistics.
\end{abstract}
\textbf{\small{}Keywords:}{\small{} permutation statistics, consecutive
patterns, Goulden\textendash Jackson cluster method, Malvenuto\textendash Reutenauer
algebra, shuffle-compatibility}{\let\thefootnote\relax}\tableofcontents{}

\section{Introduction}

Let $\mathfrak{S}_{n}$ denote the symmetric group of permutations
on the set $[n]\coloneqq\{1,2,\dots,n\}$ (where $\mathfrak{S}_{0}$
consists of the empty permutation), and let $\mathfrak{S}\coloneqq\bigsqcup_{n=0}^{\infty}\mathfrak{S}_{n}$.
We write permutations in one-line notation\textemdash that is, $\pi=\pi_{1}\pi_{2}\cdots\pi_{n}$\textemdash and
the $\pi_{i}$ are called \textit{letters} of $\pi$. The \textit{length}
of $\pi$ is the number of letters in $\pi$, so that $\pi$ has length
$n$ whenever $\pi\in\mathfrak{S}_{n}$.

For a sequence of distinct integers $w$, the \textit{standardization}
of $w$\textemdash denoted $\std(w)$\textemdash is defined to be
the permutation in $\mathfrak{S}$ obtained by replacing the smallest
letter of $w$ with 1, the second smallest with 2, and so on. As an
example, we have $\std(73184)=42153$. Given permutations $\pi\in\mathfrak{S}_{n}$
and $\sigma\in\mathfrak{S}_{m}$, we say that $\pi$ \textit{contains}
$\sigma$ (as a \textit{consecutive pattern}) if $\std(\pi_{i}\pi_{i+1}\cdots\pi_{i+m-1})=\sigma$
for some $i\in[n-m+1]$, and in this case we call $\pi_{i}\pi_{i+1}\cdots\pi_{i+m-1}$
an \textit{occurrence} of $\sigma$ (as a consecutive pattern) in
$\pi$. For instance, the permutation $315497628$ has three occurrences
of the consecutive pattern $213$, namely $315$, $549$, and $628$.
On the other hand, $137258469$ has no occurrences of $213$.

Let $\occ_{\sigma}(\pi)$ denote the number of occurrences of $\sigma$
in $\pi$. If $\occ_{\sigma}(\pi)=0$, then we say that $\pi$ \textit{avoids}
$\sigma$ (as a consecutive pattern). If $\Gamma\subseteq\mathfrak{S}$,
then we let $\mathfrak{S}_{n}(\Gamma)$ denote the subset of permutations
in $\mathfrak{S}_{n}$ avoiding every permutation in $\Gamma$ as
a consecutive pattern. When $\Gamma$ consists of a single permutation
$\sigma$, we shall simply write $\mathfrak{S}_{n}(\sigma)$ as opposed
to $\mathfrak{S}_{n}(\{\sigma\})$. (We use the same convention for
other notations involving a set $\Gamma$ of permutations when $\Gamma$
is a singleton.) As observed earlier, we have $137258469\in\mathfrak{S}_{9}(213)$.

For the rest of this paper, the notions of occurrence and avoidance
of patterns in permutations always refer to consecutive patterns unless
otherwise stated.

The study of consecutive patterns in permutations, initiated by Elizalde
and Noy \cite{Elizalde2003} in 2003, extends the study of classical
patterns in permutations originating in the work of Simion and Schmidt
\cite{Simion1985}. Consecutive patterns in permutations are analogous
to consecutive subwords in words, where repetition of letters is allowed.
In the latter realm, the cluster method of Goulden and Jackson \cite{Goulden1979}
provides a very general formula expressing the generating function
for words by occurrences of prescribed subwords in terms of a ``cluster
generating function'', which is easier to compute. By setting the
variable keeping track of occurrences to zero, this yields a powerful
approach for counting words avoiding a prescribed set of subwords.
In 2012, Elizalde and Noy \cite{Elizalde2012} adapted the Goulden\textendash Jackson
cluster method to the setting of permutations, which they used to
obtain differential equations satisfied by $\omega_{\sigma}(s,x)=(\sum_{n=0}^{\infty}\sum_{\pi\in\mathfrak{S}_{n}}s^{\occ_{\sigma}(\pi)}x^{n}/n!)^{-1}$
for various families of consecutive patterns $\sigma$, including
``monotone patterns'', ``chain patterns'', and ``non-overlapping patterns''.
Solving these differential equations for $\omega_{\sigma}(s,x)$ then
allows one to count permutations by the number of occurrences of $\sigma$.

Over the past decade, Elizalde and Noy's adaptation of the cluster
method for permutations has become a standard tool in the study of
consecutive patterns; see \cite{Beaton2017,Crane2018,Dwyer2018,Elizalde2013,Elizalde2016,Lee2019}
for a selection of references. One recent development is a $q$-analogue
of the cluster method for permutations which also keeps track of the
inversion number statistic. This $q$-cluster method, due to Elizalde,
was first mentioned in his survey \cite{Elizalde2016} on consecutive
patterns, and was applied to monotone patterns and non-overlapping
patterns by Crane, DeSalvo, and Elizalde \cite{Crane2018} in their
study of the Mallows distribution.

To explain the philosophy which guides our work, let us briefly discuss
a paper by Josuat-Verg\`es, Novelli, and Thibon \cite{Josuat-Verges2012},
in which the authors study alternating permutations (and their analogues
in other Coxeter groups) from the perspective of combinatorial Hopf
algebras. Their starting point is Andr\'e's \cite{Andre1881} famous
exponential generating function $\sec x+\tan x$ for the number of
alternating permutations. The authors note that Andr\'e's formula
has a natural lifting in the Malvenuto\textendash Reutenauer algebra
$\mathbf{FQSym}$, a Hopf algebra whose basis elements correspond
to permutations and whose multiplication encodes ``shifted concatenation''
of permutations. They then recover Andr\'e's formula by applying
a certain homomorphism $\phi$ to its lifting in $\mathbf{FQSym}$,
and in their words:
\begin{quote}
``Such a proof is not only illuminating, it says much more than the
original statement. For example, one can now replace $\phi$ by more
complicated morphisms, and obtain generating functions for various
statistics on alternating permutations.''
\end{quote}
A similar approach to permutation enumeration was taken in a series
of papers by Gessel and the present author \cite{Gessel2019,Gessel2014,Zhuang2016,Zhuang2017},
but instead utilizing homomorphisms on noncommutative symmetric functions.

The main result of this present paper is an analogous lifting of the
Goulden\textendash Jackson cluster method for permutations to the
Malvenuto\textendash Reutenauer algebra. Since the basis elements
of the Malvenuto\textendash Reutenauer algebra correspond to permutations,
our cluster method in $\mathbf{FQSym}$ is in a sense the most general
cluster method possible for permutations. By applying the same homomorphism
$\phi$ used by Josuat-Verg\`es\textendash Novelli\textendash Thibon
to our generalized cluster method, we can recover Elizalde and Noy's
cluster method for permutations, and we can use another homomorphism
to recover Elizalde's $q$-analogue. We also construct other homomorphisms
which lead to new specializations of our cluster method that can be
used to count permutations by occurrences of prescribed patterns while
keeping track of other permutation statistics.

\subsection{Permutation statistics}

The permutation statistics that we shall consider are the ``inverses''
of several classical permutation statistics related to descents and
peaks: the \textit{descent number} $\des$, the \textit{major index}
$\maj$, the \textit{comajor index} $\comaj$, the \textit{peak number}
$\pk$, and the \textit{left peak number} $\lpk$. We define these
statistics below.
\begin{itemize}
\item We call $i\in[n-1]$ a \textit{descent} of $\pi\in\mathfrak{S}_{n}$
if $\pi_{i}>\pi_{i+1}$. Then $\des(\pi)$ is defined to be the number
of descents of $\pi$, and $\maj(\pi)$ the sum of all descents of
$\pi$. In other words, if $\Des(\pi)\coloneqq\{\,i\in[n-1]:i\text{ is a descent of }\pi\,\}$\textemdash that
is, $\Des(\pi)$ is the \textit{descent set} of $\pi$\textemdash then
\[
\des(\pi)\coloneqq\left|\Des(\pi)\right|\quad\text{and\ensuremath{\quad}}\maj(\pi)\coloneqq\sum_{i\in\Des(\pi)}i.
\]
The comajor index $\comaj$ is a variant of the major index $\maj$,
and is defined by
\begin{equation}
\comaj(\pi)\coloneqq\sum_{i\in\Des(\pi)}(n-i)=n\text{\ensuremath{\des}}(\pi)-\text{\ensuremath{\maj}}(\pi).\label{e-comajdesmaj}
\end{equation}
\item We call $i\in\{2,3,\dots,n-1\}$ a \textit{peak} of $\pi\in\mathfrak{S}_{n}$
if $\pi_{i-1}<\pi_{i}>\pi_{i+1}$. Then $\pk(\pi)$ is defined to
be the number of peaks of $\pi$.
\item We call $i\in[n-1]$ a \textit{left peak} of $\pi\in\mathfrak{S}_{n}$
if $i$ is a peak of $\pi$, or if $i=1$ and $i$
is a descent of $\pi$. Then $\lpk(\pi)$ is defined to be the number
of left peaks of $\pi$.\footnote{Equivalently, $\lpk(\pi)$ is the number of peaks of the permutation
$0\pi$ obtained by prepending 0 to $\pi$.}
\end{itemize}
For example, if $\pi=72163584$, then we have $\Des(\pi)=\{1,2,4,7\}$,
$\des(\pi)=4$, $\maj(\pi)=14$, $\comaj(\pi)=18$, $\pk(\pi)=2$,
and $\lpk(\pi)=3$. We note that, in the language of consecutive patterns,
descents correspond to occurrences of $21$ and peaks correspond to
occurrences of $132$ and $231$.

Given a permutation statistic $\st$, we define its \textit{inverse
statistic} $\ist$ by $\ist(\pi)\coloneqq\st(\pi^{-1})$. Continuing
with the example from above, the inverse of $\pi$ is $\pi^{-1}=32586417$,
so we have $\iDes(\pi)=\{1,4,5,6\}$, $\ides(\pi)=4$, $\imaj(\pi)=16$,
$\icomaj(\pi)=16$, $\ipk(\pi)=1$, and $\ilpk(\pi)=2$. While $\st$
and $\ist$ are obviously equidistributed over $\mathfrak{S}_{n}$,
it is worth studying the joint distribution of $\ist$ and other permutation
statistics over $\mathfrak{S}_{n}$, or the distribution of $\ist$
over restricted sets of permutations (such as pattern avoidance classes).
For instance, Garsia and Gessel \cite{Garsia1979} studied the joint
distribution of $\des$, $\ides$, $\maj$, and $\imaj$ over $\mathfrak{S}_{n}$.

Let $\Gamma$ be a set of consecutive patterns and $\occ_{\Gamma}(\pi)$
the number of occurrences in $\pi$ of patterns in $\Gamma$. In this
paper, we will consider the polynomials {\allowdisplaybreaks 
\begin{align*}
A_{\Gamma,n}^{(\ides,\imaj)}(s,t,q) & \coloneqq\sum_{\pi\in\mathfrak{S}_{n}}s^{\occ_{\Gamma}(\pi)}t^{\ides(\pi)+1}q^{\imaj(\pi)},\\
A_{\Gamma,n}^{(\ides,\icomaj)}(s,t,q) & \coloneqq\sum_{\pi\in\mathfrak{S}_{n}}s^{\occ_{\Gamma}(\pi)}t^{\ides(\pi)+1}q^{\icomaj(\pi)},\\
A_{\Gamma,n}^{\ides}(s,t) & \coloneqq\sum_{\pi\in\mathfrak{S}_{n}}s^{\occ_{\Gamma}(\pi)}t^{\ides(\pi)+1},\\
P_{\Gamma,n}^{\ipk}(s,t) & \coloneqq\sum_{\pi\in\mathfrak{S}_{n}}s^{\occ_{\Gamma}(\pi)}t^{\ipk(\pi)+1},\text{ and}\\
P_{\Gamma,n}^{\ilpk}(s,t) & \coloneqq\sum_{\pi\in\mathfrak{S}_{n}}s^{\occ_{\Gamma}(\pi)}t^{\ilpk(\pi)}
\end{align*}
}where $n\geq1$, and with each of these polynomials defined to be
1 when $n=0$. These polynomials give the joint distribution of the
occurrence statistic $\occ_{\Gamma}$ along with each of the statistics
$(\ides,\imaj)$, $(\ides,\icomaj)$, $\ides$, $\ipk$, and $\ilpk$.
Setting $s=0$ in any of these polynomials then gives the distribution
of the corresponding statistic over the pattern avoidance class $\mathfrak{S}_{n}(\Gamma)$.
For convenience, let us define $A_{\Gamma,n}^{(\ides,\imaj)}(t,q)\coloneqq A_{\Gamma,n}^{(\ides,\imaj)}(0,t,q)$
and the polynomials $A_{\Gamma,n}^{(\ides,\icomaj)}(t,q)$, $A_{\Gamma,n}^{\ides}(t)$,
$P_{\Gamma,n}^{\ipk}(t)$, and $P_{\Gamma,n}^{\ilpk}(t)$ analogously.

The reason why we consider the statistics $(\ides,\imaj)$, $(\ides,\icomaj)$,
$\ides$, $\ipk$, and $\ilpk$ is because they are inverses of ``shuffle-compatible''
statistics. Roughly speaking, a permutation statistic $\st$ is shuffle-compatible
if the distribution of $\st$ over the set of shuffles of two permutations
$\pi$ and $\sigma$ depends only on $\st(\pi)$, $\st(\sigma)$,
and the lengths of $\pi$ and $\sigma$. (See Section 2.3 for precise
definitions.) If $\st$ is shuffle-compatible and is a coarsening
of the descent set, then $\st$ induces a quotient of the algebra
$\mathrm{QSym}$ of quasisymmetric functions, denoted ${\mathcal A}_{\st}$.
By composing the quotient map from $\mathrm{QSym}$ to ${\mathcal A}_{\st}$
with the canonical surjection from $\mathbf{FQSym}$ to $\mathrm{QSym}$,
we obtain a homomorphism on $\mathbf{FQSym}$ which can be used to
count permutations by the corresponding inverse statistic. Applying
these homomorphisms to our generalized cluster method in $\mathbf{FQSym}$
yields specializations that refine by the statistics $(\ides,\icomaj)$,
$\ides$, $\ipk$, and $\ilpk$.\footnote{We do not explicitly give a specialization for $(\ides,\imaj)$, but
one can be obtained using the one for $(\ides,\icomaj)$ and the formula
$\text{\ensuremath{\imaj}}(\pi)=n\text{\ensuremath{\ides}}(\pi)-\text{\ensuremath{\icomaj}}(\pi)$,
which is equivalent to (\ref{e-comajdesmaj}).}

\subsection{Outline}

The structure of this paper is as follows. Section 2 is devoted to
background material. We first give a brief expository account of the
Goulden\textendash Jackson cluster method, both for words and for
permutations. Then, we define quasisymmetric functions and the Malvenuto\textendash Reutenauer
algebra, and review some basic symmetries on permutations (reversal,
complementation, and reverse-complementation) which will play a role
in our work.

The focus of Section 3 is on our main result, the cluster method in
Malvenuto\textendash Reutenauer. We prove our generalized cluster
method and show how it specializes to Elizalde and Noy's cluster method
for permutations as well as its $q$-analogue. In this section, we
also use the theory of shuffle-compatibility to construct homomorphisms
which we then use to obtain further specializations of our generalized
cluster method for the statistics $(\ides,\icomaj)$, $\ides$, $\ipk$,
and $\ilpk$.

In Sections 4 and 5, we apply our general results from Section 3 to
produce formulas for the polynomials $A_{\sigma,n}^{\ides}(s,t)$,
$P_{\sigma,n}^{\ipk}(s,t)$, and $P_{\sigma,n}^{\ilpk}(s,t)$\textemdash and
their $s=0$ evaluations\textemdash where $\sigma$ is a specific
type of consecutive pattern. Section 4 focuses on \textit{monotone
patterns}, i.e., the patterns $12\cdots m$ and $m\cdots21$. Section
5 focuses on the patterns $12\cdots(a-1)(a+1)a(a+2)(a+3)\cdots m$
where $m\geq5$ and $2\leq a\leq m-2$; these patterns were considered
in \cite{Elizalde2012} as a subfamily of ``chain patterns'', and
here we call them \textit{transpositional patterns} because $12\cdots(a-1)(a+1)a(a+2)(a+3)\cdots m$
is precisely the elementary transposition $(a,a+1)$. Most of our
formulas involve the Hadamard product operation on formal power series,
although some ``Hadamard product-free'' formulas are obtained for
monotone patterns. In the case of monotone patterns, we also give
a formula for counting $12\cdots m$-avoiding permutations by inverse
descent number and inverse major index. 

We conclude this paper in Section 6 with a brief discussion of ongoing
work and future directions of research. See \cite{Zhuang2022} for an extended abstract summarizing the results of this paper, as well as \cite{Zhuang2021a} for
proofs of two observations (Claims \ref{cl-ipk} and \ref{cl-ilpk})
which are left unproven here.

\section{Preliminaries}

\subsection{The cluster method for words}

We first introduce the Goulden\textendash Jackson cluster method
for words, which we will use to prove our lifting of the cluster method
for permutations to the Malvenuto\textendash Reutenauer algebra. The
exposition in this section follows that in \cite{Zhuang2018}.

For a finite or countably infinite set $A$, let $A^{*}$ be the set
of all finite sequences of elements of $A$, including the empty sequence.
We call $A$ an \textit{alphabet}, the elements of $A$ \textit{letters},
and the elements of $A^{*}$ \textit{words}. The \textit{length} $\left|w\right|$
of a word $w\in A^{*}$ is the number of letters in $w$. For $v,w\in A^{*}$,
we say that $v$ is a \textit{subword} of $w$ if $w=uvu^{\prime}$
for some $u,u^{\prime}\in A^{*}$, and in this case we also say $w$
\textit{contains} $v$ and that $v$ is an \textit{occurrence} of
$w$. The \textit{total algebra} of $A^{*}$ over $\mathbb{Q},$ denoted
$\mathbb{Q}\langle\langle A^{*}\rangle\rangle$, is the $\mathbb{Q}$-algebra
of formal sums of words in $A^{*}$ where multiplication is the concatenation
product.

Given a word $w=w_{1}w_{2}\cdots w_{n}\in A^{*}$ and a set $B\subseteq A^{*}$,
we say that $(i,v)$ is a \textit{marked occurrence} of $v\in B$
in $w$ if 
\[
v=w_{i}w_{i+1}\cdots w_{i+\left|v\right|-1},
\]
that is, $v$ is a subword of $w$ starting at position $i$. Moreover,
we say that $(w,T)$ is a \textit{marked word} on $w$ (with respect
to $B$) if $w\in A^{*}$ and $T$ is a set of some marked occurrences
in $w$ of words in $B$.

To illustrate, suppose that $A=\{a,b,c\}$ and $B=\{cab,bc\}$. Then
\begin{equation}
(cabcabbca,\{(1,cab),(3,bc),(7,bc)\}),\label{e-mkwrd}
\end{equation}
is a marked word on $w=cabcabbca$ with respect to $B$. Informally,
we will display a marked word $(w,T)$ as the word $w$ with the marked
occurrences in $T$ circled, so that (\ref{e-mkwrd}) is displayed
as

\begin{center}
\begin{tikzpicture}

\node at (0,0) {$c\;a\;b\;c\;a\;b\;b\;c\;a\;.$};

\draw[red] (-1.04,0) ellipse (13bp and 9bp);
\draw[red] (-0.58,0) ellipse (9bp and 9bp);
\draw[red] (0.64,0) ellipse (9bp and 9bp);

\end{tikzpicture}
\end{center}

We define the concatenation of two marked words in the obvious way.
For example, (\ref{e-mkwrd}) can be obtained by concatenating $(cabca,\{(1,cab),(3,bc)\})$
and $(bbca,\{(2,bc)\})$, i.e., \begin{center}
\begin{tikzpicture}

\node at (0,0) {$c\;a\;b\;c\;a\quad$ and $\quad b\;b\;c\;a\;.$};

\draw[red] (-1.85,0) ellipse (13bp and 9bp);
\draw[red] (-1.39,0) ellipse (9bp and 9bp);
\draw[red] (1.47,0) ellipse (9bp and 9bp);

\end{tikzpicture}
\end{center}

A marked word is called a \textit{cluster} if it is not a concatenation
of two nonempty marked words. (In particular, we will call a cluster
with respect to $B$ a $B$-\textit{cluster}.) So, (\ref{e-mkwrd})
is not a cluster, but\begin{center}
\begin{tikzpicture}

\node at (0,0) {$b\;c\;a\;b\;c\;a\;b$};

\draw[red] (-0.78,0) ellipse (9bp and 9bp);
\draw[red] (-0.32,0) ellipse (13bp and 9bp);
\draw[red] (0.15,0) ellipse (9bp and 9bp);
\draw[red] (0.63,0) ellipse (13bp and 9bp);

\end{tikzpicture}
\end{center}

\noindent is a cluster.

For a word $w\in A^{*}$, let $\occ_{B}(w)$ be the number of occurrences
in $w$ of words in $B$ and let $C_{B,w}$ be the set of all $B$-clusters
on $w$. If $c$ is a $B$-cluster, then we let $\mk_{B}(c)$ be the
number of marked occurrences in $c$. Define
\[
F_{B}(s)\coloneqq\sum_{w\in A^{*}}ws^{\occ_{B}(w)}\quad\text{and}\quad R_{B}(s)\coloneqq\sum_{w\in A^{*}}w\sum_{c\in C_{B,w}}s^{\mk_{B}(c)},
\]
so that $F_{B}(s)$ is the generating function for words in $A^{*}$
by the number of occurrences of words in $B$, and $R_{B}(s)$ is
the generating function for $B$-clusters by the number of marked
occurrences. Both $F_{B}(s)$ and $R_{B}(s)$ are elements of the
formal power series algebra $\mathbb{Q}\langle\langle A^{*}\rangle\rangle[[s]]$,
so the variable $s$ commutes with letters in $A$ (but the letters
in $A$ do not commute with each other).
\begin{thm}[Cluster method for words]
\label{t-gjcm} Let $A$ be an alphabet and let $B\subseteq A^{*}$
be a set of words, each of length at least 2. Then
\[
F_{B}(s)=\bigg(1-\sum_{a\in A}a-R_{B}(s-1)\bigg)^{-1}.
\]
\end{thm}

This is a noncommutative version of the original cluster method of
Goulden and Jackson, but the proofs are essentially the same; see,
e.g., \cite[Theorem 1]{Zhuang2018} for details.

\subsection{The cluster method for permutations}

Next, we describe Elizalde and Noy's \cite{Elizalde2012} adaptation
of the cluster method for permutations, as well as its $q$-analogue
which refines by the inversion number. The terms \textit{marked occurrence},
\textit{marked permutation}, \textit{concatenation}, and \textit{cluster}
are defined for permutations in the analogous way as for words, but
with the notion of word containment replaced by permutation containment
(in the sense of consecutive patterns). It is worth pointing out that,
unlike concatenation of marked words, concatenation of marked permutations
is not unique. For instance, both\begin{center}
\begin{tikzpicture}

\node at (0,0) {$3\;2\;1\;4\;6\;7\;8\;5\;9\quad$ and $\quad 7\;5\;1\;8\;3\;4\;6\;2\;9$};

\draw[red] (-2.94,0) ellipse (14bp and 11bp);
\draw[red] (-1.94,0) ellipse (14bp and 11bp);
\draw[red] (-1.29,0) ellipse (14bp and 11bp);

\draw[red] (1.63,0) ellipse (14bp and 11bp);
\draw[red] (2.62,0) ellipse (14bp and 11bp);
\draw[red] (3.26,0) ellipse (14bp and 11bp);

\end{tikzpicture}
\end{center}are concatenations of \begin{center}
\begin{tikzpicture}

\node at (0,0) {$3\;2\;1\;4\quad$ and $\quad 2\;3\;4\;1\;5\;.$};

\draw[red] (-1.56,0) ellipse (14bp and 11bp);
\draw[red] (1.03,0) ellipse (14bp and 11bp);
\draw[red] (1.7,0) ellipse (14bp and 11bp);

\end{tikzpicture}
\end{center}However, this does not make a difference in defining clusters for
permutations or in adapting the cluster method to the setting of permutations.

Let $\Gamma\subseteq\mathfrak{S}$. Recall that $\occ_{\Gamma}(\pi)$
is the number of occurrences in $\pi$ of patterns in $\Gamma$, and
let $C_{\Gamma,\pi}$ be the set of all $\Gamma$-clusters on $\pi$.
If $c$ is a $\Gamma$-cluster, let $\mk_{\Gamma}(c)$ be the number
of marked occurrences in $c$. Define
\begin{align*}
F_{\Gamma}(s,x) & \coloneqq\sum_{n=0}^{\infty}\sum_{\pi\in\mathfrak{S}_{n}}s^{\occ_{\Gamma}(\pi)}\frac{x^{n}}{n!}\quad\text{and}\\
R_{\Gamma}(s,x) & \coloneqq\sum_{n=0}^{\infty}\sum_{\pi\in\mathfrak{S}_{n}}\sum_{c\in C_{\Gamma,\pi}}s^{\mk_{\Gamma}(c)}\frac{x^{n}}{n!}=\sum_{n=0}^{\infty}\sum_{k=0}^{\infty}r_{\Gamma,n,k}s^{k}\frac{x^{n}}{n!}
\end{align*}
where $r_{\Gamma,n,k}$ is the number of $\Gamma$-clusters of length
$n$ with $k$ marked occurrences.
\begin{thm}[Cluster method for permutations]
\label{t-gjcmperm}Let $\Gamma\subseteq\mathfrak{S}$ be a set of
permutations, each of length at least 2. Then
\[
F_{\Gamma}(s,x)=(1-x-R_{\Gamma}(s-1,x))^{-1}.
\]
\end{thm}

Elizalde and Noy give Theorem \ref{t-gjcmperm} in the special case
where $\Gamma$ consists of a single pattern \cite[Theorem 1.1]{Elizalde2012},
but in Section 3 we will recover this more general result from our
cluster method in the Malvenuto\textendash Reutenauer algebra.

The $n$th $q$-\textit{factorial} $[n]_{q}!$ is defined by 
\[
[n]_{q}!\coloneqq(1+q)(1+q+q^{2})\cdots(1+q+\cdots+q^{n-1})
\]
for $n\geq1$ and $[0]_{q}!\coloneqq1$. Later, we will also need
the $q$-\textit{binomial coefficient} defined by 
\[
{\binom{n}{k}}_{\!\!q}\coloneqq\frac{[n]_{q}!}{[k]_{q}!\,[n-k]_{q}!}
\]
for all $n\geq0$ and $0\leq k\leq n$. 

We say that $(i,j)\in[n]^{2}$ is an \textit{inversion} of $\pi\in\mathfrak{S}_{n}$
if $i<j$ and $\pi_{i}>\pi_{j}$. Let $\inv(\pi)$ denote the number
of inversions of $\pi$. Define
\begin{alignat*}{1}
F_{\Gamma}(s,q,x) & \coloneqq\sum_{n=0}^{\infty}\sum_{\pi\in\mathfrak{S}_{n}}s^{\occ_{\Gamma}(\pi)}q^{\inv(\pi)}\frac{x^{n}}{[n]_{q}!}\quad\text{and}\\
R_{\Gamma}(s,q,x) & \coloneqq\sum_{n=0}^{\infty}\sum_{\pi\in\mathfrak{S}_{n}}q^{\inv(\pi)}\sum_{c\in C_{\Gamma,\pi}}s^{\mk_{\Gamma}(c)}\frac{x^{n}}{[n]_{q}!}=\sum_{n=0}^{\infty}\sum_{k=0}^{\infty}\sum_{j=0}^{\infty}r_{\Gamma,n,k,j}q^{j}s^{k}\frac{x^{n}}{[n]_{q}!}
\end{alignat*}
where $r_{\Gamma,n,k,j}$ is the number of $\Gamma$-clusters of length
$n$ with $k$ marked occurrences and whose underlying permutation
has $j$ inversions. The next result is \cite[Theorem 2.3]{Crane2018},
but for a set $\Gamma$ of patterns rather than a single pattern $\sigma$.
\begin{thm}[$q$-Cluster method for permutations]
\label{t-qgcjmperm}Let $\Gamma\subseteq\mathfrak{S}$ be a set of
permutations, each of length at least 2. Then
\[
F_{\Gamma}(s,q,x)=(1-x-R_{\Gamma}(s-1,q,x))^{-1}.
\]
\end{thm}

See \cite[Corollary 1]{Rawlings2007} for a related result. Like with
Theorem \ref{t-gjcmperm}, we will later recover Theorem \ref{t-qgcjmperm}
as a specialization of our generalized cluster method.

Let us give one more definition before continuing. Given $\sigma\in\mathfrak{S}_{m}$,
let
\[
O_{\sigma}\coloneqq\{\,i\in[m-1]:\std(\sigma_{i+1}\sigma_{i+2}\cdots\sigma_{m})=\std(\sigma_{1}\sigma_{2}\cdots\sigma_{m-i})\,\}
\]
be the \textit{overlap set} of $\sigma$. The notion of overlap set
is useful for characterizing $\Gamma$-clusters where $\Gamma$ consists
of a single pattern $\sigma$, and we will do this in Sections 4.1
and 5.1.

\subsection{Quasisymmetric functions and shuffle-compatibility}

A permutation in $\mathfrak{S}_{n}$ can be characterized as a word
in $[n]^{*}$ of length $n$ consisting of distinct letters. Let $\mathbb{P}$
be the set of positive integers, and let $\mathfrak{P}_{n}$ denote
the set of words in $\mathbb{P}^{*}$ of length $n$ consisting of
distinct letters\textemdash not necessarily from 1 to $n$. Also,
let $\mathfrak{P}\coloneqq\bigsqcup_{n=0}^{\infty}\mathfrak{P}_{n}$.
In this section only, we will use the term ``permutation'' to refer
more generally to elements of $\mathfrak{P}$. Observe that any statistic
$\st$ defined on permutations in $\mathfrak{S}$ can be extended
to $\mathfrak{P}$ by letting $\st(\pi)\coloneqq\st(\std(\pi))$ for
$\pi\in\mathfrak{P}$.

Every permutation in $\mathfrak{P}$ can be uniquely decomposed into
a sequence of maximal increasing consecutive subsequences, which we
call \textit{increasing runs}. Equivalently, an increasing run of
$\pi$ is a maximal consecutive subsequence containing no descents.
The \textit{descent composition} of $\pi$, denoted $\Comp(\pi)$,
is the composition whose parts are the lengths of the increasing runs
of $\pi$ in the order that they appear. For instance, the increasing
runs of $\pi=85712643$ are $8$, $57$, $126$, $4$, and $3$, so
the descent composition of $\pi$ is $\Comp(\pi)=(1,2,3,1,1)$. We
use the notations $L\vDash n$ and $\left|L\right|=n$ to indicate
that $L$ is a composition of $n$, so that $L\vDash n$ and $\left|L\right|=n$
whenever $L$ is the descent composition of a permutation in $\mathfrak{P}_{n}$.
For a composition $L=(L_{1},L_{2},\dots,L_{k})$, let $\Des(L)\coloneqq\{L_{1},L_{1}+L_{2},\dots,L_{1}+\cdots+L_{k-1}\}$.
It is easy to see that if $L$ is the descent composition of $\pi$,
then $\Des(L)$ is the descent set of $\pi$.

If $\pi\in\mathfrak{P}_{m}$ and $\sigma\in\mathfrak{P}_{n}$ are
\textit{disjoint}\textemdash that is, if they have no letters in common\textemdash then
we call $\tau\in\mathfrak{P}_{m+n}$ a \textit{shuffle} of $\pi$
and $\sigma$ if both $\pi$ and $\sigma$ are subsequences of $\tau$.
The set of shuffles of $\pi$ and $\sigma$ is denoted $S(\pi,\sigma)$.
For example, we have 
\[
S(31,25)=\{3125,3215,3251,2315,2351,2531\}.
\]

Let $x_{1},x_{2},\dots$ be commuting variables. A formal power series
$f\in\mathbb{Q}[[x_{1},x_{2},\dots]]$ of bounded degree is called
a \textit{quasisymmetric function} if for any positive integers $a_{1},a_{2},\dots,a_{k}$,
if $i_{1}<i_{2}<\cdots<i_{k}$ and $j_{1}<j_{2}<\cdots<j_{k}$ then
\[
[x_{i_{1}}^{a_{1}}x_{i_{2}}^{a_{2}}\cdots x_{i_{k}}^{a_{k}}]\,f=[x_{j_{1}}^{a_{1}}x_{j_{2}}^{a_{2}}\cdots x_{j_{k}}^{a_{k}}]\,f.
\]
Let $\mathrm{QSym}_{n}$ denote the set of quasisymmetric functions
homogeneous of degree $n$. As a vector space, $\mathrm{QSym}_{n}$
has as a basis the \textit{fundamental quasisymmetric functions} $\{F_{L}\}_{L\vDash n}$
defined by
\[
F_{L}\coloneqq\sum_{\substack{i_{1}\leq i_{2}\leq\cdots\leq i_{n}\\
i_{j}<i_{j+1}\text{ if }j\in\Des(L)
}
}x_{i_{1}}x_{i_{2}}\cdots x_{i_{n}}.
\]
If $L\vDash m$ and $K\vDash n$, then 
\begin{equation}
F_{L}F_{K}=\sum_{\tau\in S(\pi,\sigma)}F_{\Comp(\tau)}\label{e-fundprod}
\end{equation}
where $\pi$ and $\sigma$ are any disjoint permutations satisfying
$\Comp(\pi)=L$ and $\Comp(\sigma)=K$. Hence, 
\[
\mathrm{QSym}\coloneqq\bigoplus_{n=0}^{\infty}\mathrm{QSym}_{n}
\]
is a graded subalgebra of $\mathbb{Q}[[x_{1},x_{2},\dots]]$; this
is the \textit{algebra of quasisymmetric functions} (over $\mathbb{Q}$). 

Motivated by Stanley's theory of $P$-partitions, quasisymmetric functions
were first defined and studied by Gessel \cite{Gessel1984} and are
now ubiquitous in algebraic combinatorics. References on quasisymmetric
functions include \cite[Section 7.19]{Stanley2001}, \cite[Section 5]{Grinberg2020},
and \cite{Luoto2013}.

Let us now return to the product formula (\ref{e-fundprod}) for fundamental
quasisymmetric functions. In order for (\ref{e-fundprod}) to make
sense, the multiset $\{\,\Comp(\tau):\tau\in S(\pi,\sigma)\,\}$ must
only depend on the descent compositions of $\pi$ and $\sigma$, or
equivalently, $\{\,\Des(\tau):\tau\in S(\pi,\sigma)\,\}$ only depends
on $\Des(\pi)$, $\Des(\sigma)$, and the lengths of $\pi$ and $\sigma$.
More generally, a permutation statistic $\st$ is called \textit{shuffle-compatible}
if for any disjoint permutations $\pi$ and $\sigma$, the multiset
$\{\,\st(\tau):\tau\in S(\pi,\sigma)\,\}$ depends only on $\st(\pi)$,
$\st(\sigma)$, and the lengths of $\pi$ and $\sigma$. Therefore,
the descent set $\Des$ is a shuffle-compatible permutation statistic.

In \cite{Gessel2018}, Gessel and the present author develop a theory
of shuffle-compatibility for \textit{descent statistics}: statistics
$\st$ such that $\Comp(\pi)=\Comp(\sigma)$ implies $\st(\pi)=\st(\sigma)$.
The statistics $\des$, $\maj$, $\comaj$, $\pk$, and $\lpk$ are
all examples of shuffle-compatible descent statistics. If $\st$ is
a descent statistic and if $L$ is a composition, then we let $\st(L)$
denote the value of $\st$ on any permutation with descent composition
$L$. Two compositions $L$ and $K$ are called $\st$-\textit{equivalent}
if $\st(L)=\st(K)$ and $\left|L\right|=\left|K\right|$. The following
is Theorem 4.3 of \cite{Gessel2018}, and provides a necessary and
sufficient condition for a descent statistic to be shuffle-compatible.
\begin{thm}
\label{t-gzmain}A descent statistic $\st$ is shuffle-compatible
if and only if there exists a $\mathbb{Q}$-algebra homomorphism $\phi_{\st}\colon\mathrm{QSym}\rightarrow{\mathcal A}_{\st}$,
where ${\mathcal A}_{\st}$ is a $\mathbb{Q}$-algebra with basis $\{u_{\alpha}\}$
indexed by $\st$-equivalence classes $\alpha$ of compositions, such
that $\phi_{\st}(F_{L})=u_{\alpha}$ whenever $L$ is in the $\st$-equivalence
class $\alpha$.
\end{thm}

Gessel and the present author call ${\mathcal A}_{\st}$ the \textit{shuffle
algebra} of $\st$, because the basis elements $u_{\alpha}$ can be
viewed as encoding the distribution of $\st$ over shuffles of permutations.
Theorem \ref{t-gzmain} implies that ${\mathcal A}_{\st}$ is isomorphic
to a quotient of $\mathrm{QSym}$ whenever $\st$ is a shuffle-compatible
descent statistic. We will not be working with the algebras ${\mathcal A}_{\st}$
themselves, but rather with the homomorphisms $\phi_{\st}$. Note
that in the special case of the descent set, $\phi_{\Des}$ is an
isomorphism and the basis $\{u_{\alpha}\}$ of ${\mathcal A}_{\Des}$
corresponds directly to the fundamental basis of $\mathrm{QSym}$.

\subsection{The Malvenuto\textendash Reutenauer algebra}

Let $\mathbb{Q}[\mathfrak{S}]$ denote the $\mathbb{Q}$-vector space
with basis elements the permutations in $\mathfrak{S}$. The \textit{Malvenuto\textendash Reutenauer
algebra}, first defined in \cite{Malvenuto1995}, is the $\mathbb{Q}$-algebra
on $\mathbb{Q}[\mathfrak{S}]$ with the product 
\[
\pi\cdot\sigma=\sum_{\tau\in C(\pi,\sigma)}\tau
\]
where $C(\pi,\sigma)$ is the set of \textit{shifted concatenations}
of $\pi$ and $\sigma$. That is, if $\pi\in\mathfrak{S}_{m}$ and
$\sigma\in\mathfrak{S}_{n}$ then
\[
C(\pi,\sigma)\coloneqq\{\,\tau\in\mathfrak{S}_{m+n}:\std(\tau_{1}\cdots\tau_{m})=\pi\text{ and }\text{\ensuremath{\std}(\ensuremath{\tau_{m+1}\cdots\tau_{m+n}}})=\sigma\,\}.
\]
Note that the Malvenuto\textendash Reutenauer algebra is graded by
the length of the permutation, and that its identity element is the
empty permutation.

Rather than using the original construction of the Malvenuto\textendash Reutenauer
algebra as given above, we will follow the approach of Duchamp, Hivert,
and Thibon \cite{Duchamp2002}, who gave another realization of the
Malvenuto\textendash Reutenauer algebra as a subalgebra of $\mathbb{Q}\langle\langle A^{*}\rangle\rangle$
where $A$ consists of the noncommuting variables $X_{1},X_{2},\dots$.
In order to describe their construction, we must revisit the standardization
map $\std$. We extend the map $\std$ to all words on the alphabet
$\mathbb{P}$ of positive integers using the following rule: if a
letter repeats, then they are viewed as increasing from left to right.
For example, $\std(145411)=146523$. We will later use the following
fact, which is Proposition 5.3.2 of \cite{Grinberg2020}.
\begin{prop}
\label{p-std} Let $w=w_{1}w_{2}\cdots w_{n}$ be a word in $\mathbb{P}^{*}$
of length $n$, and let $\tau=\tau_{1}\tau_{2}\cdots\tau_{n}=\std(w)$.
Then $\tau$ is the unique permutation in $\mathfrak{S}_{n}$ such
that, whenever $1\leq i<j\leq n$, we have $\tau_{i}<\tau_{j}$ if
and only if $w_{i}\leq w_{j}$.
\end{prop}

For the remainder of this section, let $A=\{X_{1},X_{2},\dots\}$
where the $X_{i}$ are noncommuting variables. Given a monomial $\mathbf{X}=X_{i_{1}}X_{i_{2}}\cdots X_{i_{n}}$,
define $\std(\mathbf{X})\coloneqq\std(i_{1}i_{2}\cdots i_{n})$. Then
we associate to each permutation $\pi\in\mathfrak{S}$ an element
$\mathbf{G}_{\pi}\in\mathbb{Q}\langle\langle A^{*}\rangle\rangle$
defined by 
\[
\mathbf{G}_{\pi}\coloneqq\sum_{\substack{\mathbf{X}\in A^{*}\\
\std(\mathbf{X})=\pi
}
}\mathbf{X}.
\]
It can be shown that the $\mathbf{G}_{\pi}$ are linearly independent
and multiply by the rule 
\[
\mathbf{G}_{\pi}\mathbf{G}_{\sigma}=\sum_{\tau\in C(\pi,\sigma)}\mathbf{G}_{\tau},
\]
so $\{\mathbf{G}_{\pi}\}_{\pi\in\mathfrak{S}}$ spans a $\mathbb{Q}$-subalgebra
of $\mathbb{Q}\langle\langle A^{*}\rangle\rangle$, called the \textit{algebra
of free quasisymmetric functions} and denoted $\mathbf{FQSym}$. Since
$\pi\mapsto\mathbf{G}_{\pi}$ is clearly a $\mathbb{Q}$-algebra isomorphism
between $\mathbb{Q}[\mathfrak{S}]$ and $\mathbf{FQSym}$, we will henceforth refer to $\mathbf{FQSym}$ as the Malvenuto\textendash Reutenauer
algebra. We use $\mathbf{FQSym}$ instead of $\mathbb{Q}[\mathfrak{S}]$
because, by identifying permutations with elements of $\mathbb{Q}\langle\langle A^{*}\rangle\rangle$,
we can prove our generalized cluster method for permutations using
the cluster method for words.\footnote{It is possible to prove our generalized cluster method in $\mathbb{Q}[\mathfrak{S}]$,
and we do this in the extended abstract \cite{Zhuang2022}. While
that approach is more direct, our approach here gives a unified treatment
of the cluster method for words and the cluster method for permutations.}

The Malvenuto\textendash Reutenauer algebra $\mathbf{FQSym}$ contains
an important subalgebra related to descent sets. Given a composition
$L$, let $\mathbf{r}_{L}$ be the sum of all $\mathbf{G}_{\pi}$
for which $\pi$ has descent composition $L$; that is, let
\[
\mathbf{r}_{L}\coloneqq\sum_{\Comp(\pi)=L}\mathbf{G}_{\pi}.
\]
The $\{\mathbf{r}_{L}\}_{L\vDash n,\,n\geq0}$ is a linearly independent
set and spans a $\mathbb{Q}$-subalgebra of $\mathbf{FQSym}$ called
the \textit{algebra of noncommutative symmetric functions}, denoted
$\mathbf{Sym}$. Noncommutative symmetric functions were introduced
in the seminal paper \cite{ncsf1} of Gelfand et al., but implicitly
appeared earlier in Gessel's Ph.D.\ thesis \cite{gessel-thesis}.

Let $\iota\colon\mathbf{Sym}\rightarrow\mathbf{FQSym}$ denote the
canonical inclusion map from $\mathbf{Sym}$ to $\mathbf{FQSym}$.
There is also a natural surjection $\rho\colon\mathbf{FQSym}\rightarrow\mathrm{QSym}$
given by 
\begin{equation}
\rho(\mathbf{G}_{\pi})\coloneqq F_{\Comp(\pi^{-1})}.\label{e-cansurj}
\end{equation}
The map $\rho$ explains the name ``free quasisymmetric functions'',
as the elements of $\mathbf{FQSym}$ lift quasisymmetric functions
to a noncommutative setting. We will need $\rho$ to define the homomorphisms
on $\mathbf{FQSym}$ that we will use to study inverse statistics.

It is worth mentioning that $\mathrm{QSym}$, $\mathbf{FQSym}$, and
$\mathbf{Sym}$ are prototypical examples of combinatorial Hopf algebras,
but we only need the algebra structure in our work. See \cite{Grinberg2020}
for a survey on Hopf algebras in combinatorics, including more on
the relationship between $\mathrm{QSym}$, $\mathbf{FQSym}$, and
$\mathbf{Sym}$.

\subsection{Symmetries on permutations}

Given $\pi\in\mathfrak{S}_{n}$, we define its \textit{reverse} $\pi^{r}$
and its \textit{complement} $\pi^{c}$ by
\[
\pi^{r}\coloneqq\pi_{n}\pi_{n-1}\cdots\pi_{1}\quad\text{and}\quad\pi^{c}\coloneqq(n+1-\pi_{1})(n+1-\pi_{2})\cdots(n+1-\pi_{n}),
\]
respectively, and its \textit{reverse-complement} $\pi^{rc}$ by $\pi^{rc}\coloneqq(\pi^{r})^{c}=(\pi^{c})^{r}$. 

It is clear that\textemdash along with permutation inversion $(\pi\mapsto\pi^{-1})$\textemdash reversion,
complementation, and reverse-complementation are all involutions on
$\mathfrak{S}_{n}$, and they can be identified with rigid motions
in the dihedral group of the square acting on permutation matrices.
As such, it is easy to see that $(\pi^{-1})^{r}=(\pi^{c})^{-1}$ and
$(\pi^{-1})^{c}=(\pi^{r})^{-1}$\textemdash and hence $(\pi^{-1})^{rc}=(\pi^{rc})^{-1}$\textemdash for
all $\pi\in\mathfrak{S}$.
\begin{prop}
\label{p-rcist} For any $\pi\in\mathfrak{S}_{n}$ with $n\geq1$,
we have 
\begin{enumerate}
\item [\normalfont{(a)}] $\imaj(\pi^{r})={\binom{n}{2}}-\imaj(\pi)$,
\item [\normalfont{(b)}] $\imaj(\pi^{rc})=\icomaj(\pi)$,
\item [\normalfont{(c)}] $\ides(\pi^{r})=\ides(\pi^{c})=n-1-\ides(\pi)$,
\item [\normalfont{(d)}] $\ides(\pi^{rc})=\ides(\pi)$, and
\item [\normalfont{(e)}] $\ipk(\pi^{c})=\ipk(\pi)$.
\end{enumerate}
\end{prop}

\begin{proof}
Since $\maj(\pi^{c})={\binom{n}{2}}-\maj(\pi)$ for all $\pi\in\mathfrak{S}_{n}$
\cite[Lemma 2.5]{Dokos2012}, we have
\[
\imaj(\pi^{r})=\maj((\pi^{r})^{-1})=\maj((\pi^{-1})^{c})={\binom{n}{2}}-\maj(\pi^{-1})={\binom{n}{2}}-\imaj(\pi),
\]
which proves (a). The proofs of (b)\textendash (e) are similar, using
the identities $\maj(\pi^{rc})=\comaj(\pi)$, $\des(\pi^{r})=\des(\pi^{c})=n-1-\des(\pi)$,
$\des(\pi^{rc})=\des(\pi)$, and $\pk(\pi^{r})=\pk(\pi)$.
\end{proof}
Let us define $\Gamma^{r}\coloneqq\{\,\pi^{r}:\pi\in\Gamma\,\}$ for
a set $\Gamma\subseteq\mathfrak{S}$ of permutations, and we define
$\Gamma^{c}$ and $\Gamma^{rc}$ in the analogous way. The next proposition
tells us that if two sets of patterns $\Gamma$ and $\Delta$ are
related by one of these symmetries, then we may be able to compute
$A_{\Gamma,n}^{(\ides,\imaj)}(s,t,q)$, $A_{\Gamma,n}^{\ides}(s,t)$,
or $P_{\Gamma,n}^{\ipk}(s,t)$ using the corresponding polynomials
for $\Delta$. For example, once we obtain a generating function formula
for the polynomials $A_{12\cdots m,n}^{(\ides,\imaj)}(t,q)$ in Section
4.2, we can use this formula along with Proposition \ref{p-rcpoly}
(a) to compute the polynomials $A_{m\cdots21,n}^{(\ides,\imaj)}(t,q)$.
\begin{prop}
\label{p-rcpoly}For any $\pi\in\mathfrak{S}_{n}$ with $n\geq1$,
we have
\begin{enumerate}
\item [\normalfont{(a)}] $A_{\Gamma^{r},n}^{(\ides,\imaj)}(s,t,q)=t^{n+1}q^{\binom{n}{2}}A_{\Gamma,n}^{(\ides,\imaj)}(s,t^{-1},q^{-1})$,
\item [\normalfont{(b)}] $A_{\Gamma^{rc},n}^{(\ides,\imaj)}(s,t,q)=A_{\Gamma,n}^{(\ides,\icomaj)}(s,t,q)$,
\item [\normalfont{(c)}] $A_{\Gamma^{rc},n}^{\ides}(s,t)=A_{\Gamma,n}^{\ides}(s,t)$,
\item [\normalfont{(d)}] $A_{\Gamma^{r},n}^{\ides}(s,t)=A_{\Gamma^{c},n}^{\ides}(s,t)=t^{n+1}A_{\Gamma,n}^{\ides}(s,t^{-1})$,
and
\item [\normalfont{(e)}] $P_{\Gamma^{c},n}^{\ipk}(s,t)=P_{\Gamma,n}^{\ipk}(s,t)$.
\end{enumerate}
\end{prop}

\begin{proof}
Each of these identities follows from algebraic manipulations, Proposition
\ref{p-rcist}, and the fact that an occurrence of a pattern $\sigma$
in $\pi$ directly corresponds to an occurrence of $\sigma^{r}$ (respectively,
$\sigma^{c}$ and $\sigma^{rc}$) in $\pi^{r}$ (respectively, $\pi^{c}$
and $\pi^{rc}$). We demonstrate the proof for (a) and leave the rest
to the reader: 
\begin{align*}
t^{n+1}q^{{n \choose 2}}A_{\Gamma,n}^{(\ides,\imaj)}(s,t^{-1},q^{-1}) & =t^{n+1}q^{{n \choose 2}}\sum_{\pi\in\mathfrak{S}_{n}}s^{\occ_{\Gamma}(\pi)}(t^{-1})^{\ides(\pi)+1}(q^{-1})^{\imaj(\pi)}\\
 & =\sum_{\pi\in\mathfrak{S}_{n}}s^{\occ_{\Gamma}(\pi)}t^{n+1-(n-1-\ides(\pi^{r}))-1}q^{{n \choose 2}-\imaj(\pi)}\\
 & =\sum_{\pi\in\mathfrak{S}_{n}}s^{\occ_{\Gamma^{r}}(\pi^{r})}t^{\ides(\pi^{r})+1}q^{\imaj(\pi^{r})}\\
 & =\sum_{\pi\in\mathfrak{S}_{n}}s^{\occ_{\Gamma^{r}}(\pi)}t^{\ides(\pi)+1}q^{\imaj(\pi)}\\
 & =A_{\Gamma^{r},n}^{(\ides,\imaj)}(s,t,q).\tag*{\qedhere}
\end{align*}
\end{proof}

\section{The cluster method in Malvenuto\textendash Reutenauer}

\subsection{Main result}

Given a set of permutations $\Gamma\subseteq\mathfrak{S}$, define
\[
\bar{F}_{\Gamma}(s)\coloneqq\sum_{\pi\in\mathfrak{S}}\mathbf{G}_{\pi}s^{\occ_{\Gamma}(\pi)}\quad\text{and}\quad\bar{R}_{\Gamma}(s)\coloneqq\sum_{\pi\in\mathfrak{S}}\mathbf{G}_{\pi}\sum_{c\in C_{\Gamma,\pi}}s^{\mk_{\Gamma}(c)},
\]
which are liftings of the exponential generating functions $F_{\Gamma}(s,x)$
and $R_{\Gamma}(s,x)$ from Section 2.2.
\begin{thm}[Cluster method in $\mathbf{FQSym}$]
\label{t-gjcmfqsym} Let $\Gamma\subseteq\mathfrak{S}$ be a set
of permutations, each of length at least 2. Then
\begin{equation}
\bar{F}_{\Gamma}(s)=\Big(1-\mathbf{G}_{1}-\bar{R}_{\Gamma}(s-1)\Big)^{-1}.\label{e-fqsymgjcm}
\end{equation}
\end{thm}

We will prove Theorem \ref{t-gjcmfqsym} using Theorem \ref{t-gjcm},
the noncommutative version of the original Goulden\textendash Jackson
cluster method for words. Our proof will rely on several preliminary
lemmas, which we establish below.
\begin{lem}
\label{l-stdsubword} Let $u=u_{1}u_{2}\cdots u_{n}$ and $v=v_{1}v_{2}\cdots v_{n}$
be two words in $\mathbb{P}^{*}$. If $\std(u)=\std(v)$, then for
any $0\leq m\leq n-1$ and $1\leq k\leq n-m$, we have $\std(u_{k}u_{k+1}\cdots u_{k+m})=\std(v_{k}v_{k+1}\cdots v_{k+m})$.
\end{lem}

\begin{proof}
Recall from Proposition \ref{p-std} that if $\tau=\std(u)$, then
whenever $1\leq i<j\leq n$, we have $u_{i}\leq u_{j}$ if and only
if $\tau_{i}<\tau_{j}$. Since we are given that $\std(u)=\std(v)$,
we have $u_{i}\leq u_{j}$ if and only if $v_{i}\leq v_{j}$ for all
$1\leq i<j\leq n$. In particular, we have $u_{i}\leq u_{j}$ if and
only if $v_{i}\leq v_{j}$ for all $k\leq i<j\leq k+m$; thus $\std(u_{k}u_{k+1}\cdots u_{k+m})=\std(v_{k}v_{k+1}\cdots v_{k+m})$.
\end{proof}
For the remainder of this section, let $A$ be the set of noncommuting
variables $\{X_{1},X_{2},\dots\}$, let $M(\pi)$ be the set of monomials
in these variables whose standardization is $\pi$, and let $B=\bigsqcup_{\sigma\in\Gamma}M(\sigma)$.
\begin{lem}
\label{l-occstd} If $\mathbf{X}\in M(\pi)$, then $\occ_{B}(\mathbf{X})=\occ_{\Gamma}(\pi)$.
\end{lem}

\begin{proof}
Write $\pi=\pi_{1}\pi_{2}\cdots\pi_{n}$ and $\mathbf{X}=X_{i_{1}}X_{i_{2}}\cdots X_{i_{n}}$.
Since $\mathbf{X}\in M(\pi)$, we have that $\std(i_{1}i_{2}\cdots i_{n})=\pi$.
Suppose that $X_{i_{k}}X_{i_{k+1}}\cdots X_{i_{k+m}}$ is an occurrence
of a word from $B$, so $X_{i_{k}}X_{i_{k+1}}\cdots X_{i_{k+m}}\in M(\sigma)$
for some $\sigma\in\Gamma$ and thus 
\[
\std(i_{k}i_{k+1}\cdots i_{k+m})=\std(X_{i_{k}}X_{i_{k+1}}\cdots X_{i_{k+m}})=\sigma.
\]
Since $\std(i_{1}i_{2}\cdots i_{n})=\pi$ and $\std(i_{k}i_{k+1}\cdots i_{k+m})=\sigma$,
it follows from Lemma \ref{l-stdsubword} that $\std(\pi_{k}\pi_{k+1}\cdots\pi_{k+m})=\sigma$.
In other words, $\pi_{k}\pi_{k+1}\cdots\pi_{k+m}$ is an occurrence
of $\sigma$ in $\pi$. We can go backward to see that there is a
bijection between occurrences of words from $B$ in $\mathbf{X}$
and patterns from $\Gamma$ in $\pi$, which shows $\occ_{B}(\mathbf{X})=\occ_{\Gamma}(\pi)$.
\end{proof}
\begin{lem}
\label{l-mkstd} If $\mathbf{X}\in M(\pi)$, then $\sum_{c\in C_{B,\mathbf{X}}}s^{\mk_{B}(c)}=\sum_{c\in C_{\Gamma,\pi}}s^{\mk_{\Gamma}(c)}$.
\end{lem}

\begin{proof}
Similar reasoning as above can be used to show that there is a bijection
between $B$-clusters on $\mathbf{X}$ and $\Gamma$-clusters on $\pi$
which preserves the number (and positions) of marked occurrences;
we omit the details.
\end{proof}
We are now ready to prove Theorem \ref{t-gjcmfqsym}.
\begin{proof}[Proof of Theorem \ref{t-gjcmfqsym}]
As a consequence of Lemma \ref{l-occstd}, we have that 
\[
\sum_{\pi\in\mathfrak{S}}\mathbf{G}_{\pi}s^{\occ_{\Gamma}(\pi)}=\sum_{\pi\in\mathfrak{S}}\mathbf{G}_{\pi}s^{\occ_{B}(\pi)}
\]
where $\occ_{B}(\pi)\coloneqq\occ_{B}(\mathbf{X})$ for any $\mathbf{X}\in M(\pi)$.
Hence 
\begin{alignat}{1}
\bar{F}_{\Gamma}(s) & =\sum_{\pi\in\mathfrak{S}}\mathbf{G}_{\pi}s^{\occ_{\Gamma}(\pi)}\nonumber \\
 & =\sum_{\pi\in\mathfrak{S}}\mathbf{G}_{\pi}s^{\occ_{B}(\pi)}\nonumber \\
 & =\sum_{\mathbf{X}\in A^{*}}\mathbf{X}s^{\occ_{B}(\mathbf{X})}\nonumber \\
 & =F_{B}(s).\label{e-fqsymF}
\end{alignat}
Similarly, Lemma \ref{l-mkstd} implies 
\[
\sum_{\pi\in\mathfrak{S}}\mathbf{G}_{\pi}\sum_{c\in C_{\Gamma,\pi}}s^{\mk_{\Gamma}(c)}=\sum_{\pi\in\mathfrak{S}}\mathbf{G}_{\pi}\sum_{c\in C_{B,\pi}}s^{\mk_{B}(c)}
\]
where $C_{B,\pi}\coloneqq C_{B,\mathbf{X}}$ for any $\mathbf{X}\in M(\pi)$.
Thus 
\begin{align}
\bar{R}_{\Gamma}(s) & =\sum_{\pi\in\mathfrak{S}}\mathbf{G}_{\pi}\sum_{c\in C_{\Gamma,\pi}}s^{\mk_{\Gamma}(c)}\nonumber \\
 & =\sum_{\pi\in\mathfrak{S}}\mathbf{G}_{\pi}\sum_{c\in C_{B,\pi}}s^{\mk_{B}(c)}\nonumber \\
 & =\sum_{\mathbf{X}\in A^{*}}\mathbf{X}\sum_{c\in C_{B,\mathbf{X}}}s^{\mk_{B}(c)}\nonumber \\
 & =R_{B}(s).\label{e-fqsymR}
\end{align}
Finally, we use Theorem \ref{t-gjcm} along with Equations (\ref{e-fqsymF})
and (\ref{e-fqsymR}) to conclude
\begin{align*}
\bar{F}_{\Gamma}(s) & =F_{B}(s)\\
 & =\Big(1-\mathbf{G}_{1}-R_{B}(s-1)\Big)^{-1}\\
 & =\Big(1-\mathbf{G}_{1}-\bar{R}_{\Gamma}(s-1)\Big)^{-1}.\qedhere
\end{align*}
\end{proof}

\subsection{Two basic homomorphisms}

We now demonstrate how Elizalde and Noy's cluster method for permutations,
as well as Elizalde's $q$-analogue, can be recovered from the cluster
method in $\mathbf{FQSym}$.

Given $\pi\in\mathfrak{S}_{n}$, define the maps $\Psi\colon\mathbf{FQSym}\rightarrow\mathbb{Q}[[x]]$
and $\Psi_{q}\colon\mathbf{FQSym}\rightarrow\mathbb{Q}[[q,x]]$ by
\[
\Psi(\mathbf{G}_{\pi})\coloneqq\frac{x^{n}}{n!}\qquad\mathrm{and}\qquad\Psi_{q}(\mathbf{G}_{\pi})\coloneqq q^{\inv(\pi)}\frac{x^{n}}{[n]_{q}!}
\]
and extending linearly.
\begin{prop}
The maps $\Psi$ and $\Psi_{q}$ are $\mathbb{Q}$-algebra homomorphisms.
\end{prop}

\begin{proof}
Let $\pi\in\mathfrak{S}_{m}$ and $\sigma\in\mathfrak{S}_{n}$. Using
the multiplication rule for the $\mathbf{G}_{\pi}$, the definition
of the map $\Psi_{q}$, and the identity
\[
\sum_{\tau\in C(\pi,\sigma)}q^{\inv(\tau)}=q^{\inv(\pi)+\inv(\sigma)}{m+n \choose n}_{\!\!q}
\]
(see \cite[Lemma 2.1]{Crane2018}), we obtain {\allowdisplaybreaks
\begin{align*}
\Psi_{q}(\mathbf{G}_{\pi}\mathbf{G}_{\sigma}) & =\sum_{\tau\in C(\pi,\sigma)}\Psi_{q}(\mathbf{G}_{\tau})\\
 & =\sum_{\tau\in C(\pi,\sigma)}q^{\inv(\tau)}\frac{x^{m+n}}{[m+n]_{q}!}\\
 & =q^{\inv(\pi)+\inv(\sigma)}{m+n \choose n}_{\!\!q}\frac{x^{m+n}}{[m+n]_{q}!}\\
 & =q^{\inv(\pi)}\frac{x^{m}}{[m]_{q}!}q^{\inv(\sigma)}\frac{x^{n}}{[n]_{q}!}\\
 & =\Psi_{q}(\mathbf{G}_{\pi})\Psi_{q}(\mathbf{G}_{\sigma});
\end{align*}
hence, $\Psi_{q}$ is a homomorphism. Setting $q=1$ above yields
\[
\Psi(\mathbf{G}_{\pi}\mathbf{G}_{\sigma})=\sum_{\tau\in C(\pi,\sigma)}\frac{x^{m+n}}{(m+n)!}=\frac{x^{m}}{m!}\frac{x^{n}}{n!}=\Psi(\mathbf{G}_{\pi})\Psi(\mathbf{G}_{\sigma}),
\]
which shows that $\Psi$ is a homomorphism as well.}
\end{proof}
The homomorphism $\Psi$ is precisely the homomorphism $\phi$ of
Josuat-Verg\`es, Novelli, and Thibon mentioned in the introduction
of this paper. It is easy to see that upon applying $\Psi$ to our
cluster method in $\mathbf{FQSym}$ (Theorem \ref{t-gjcmfqsym}),
we recover Elizalde and Noy's cluster method for permutations (Theorem
\ref{t-gjcmperm}). Applying $\Psi_{q}$ instead yields a proof of
Elizalde's $q$-cluster method (Theorem \ref{t-qgcjmperm}).

We also note that $\Psi$ and $\Psi_{q}$ are closely related to two
homomorphisms, denoted $\Phi$ and $\Phi_{q}$, which appear in \cite{Zhuang2017}.
(The map $\Phi$ also appears in \cite{Gessel2019,Gessel2014,Zhuang2016}.)
These two homomorphisms are defined on the algebra $\mathbf{Sym}$
of noncommutative symmetric functions by the formulas
\[
\Phi(\mathbf{h}_{n})\coloneqq\frac{x^{n}}{n!}\quad\text{and}\quad\Phi_{q}(\mathbf{h}_{n})\coloneqq\frac{x^{n}}{[n]_{q}!},
\]
where $\mathbf{h}_{n}\coloneqq\sum_{i_{1}\leq\cdots\leq i_{n}}X_{i_{1}}X_{i_{2}}\cdots X_{i_{n}}=\mathbf{G}_{12\cdots n}.$
In fact, $\Phi$ and $\Phi_{q}$ are related to our homomorphisms
$\Psi$ and $\Psi_{q}$ by 
\[
\Phi=\Psi\circ\iota\quad\text{and}\quad\Phi_{q}=\Psi_{q}\circ\iota
\]
where $\iota$ is the canonical inclusion from $\mathbf{Sym}$ to
$\mathbf{FQSym}$.

\subsection{A note on the Hadamard product}

Our next goal is to define a family of homomorphisms on $\mathbf{FQSym}$
that can be used to produce other specializations of our generalized
cluster method. Our starting point is Theorem \ref{t-gzmain}, which
states that every shuffle-compatible descent statistic $\st$ gives
rise to a homomorphism $\phi_{\st}$ from the algebra $\mathrm{QSym}$
of quasisymmetric functions to the shuffle algebra ${\mathcal A}_{\st}$
of $\st$. Many of these algebras ${\mathcal A}_{\st}$ can be characterized
as subalgebras of various formal power series algebras in which the
multiplication is the ``Hadamard product'' in a variable $t$, which
we define below.

The operation of \textit{Hadamard product} $*$ on formal power series
in $t$ is defined by

\[
\Big(\sum_{n=0}^{\infty}a_{n}t^{n}\Big)*\Big(\sum_{n=0}^{\infty}b_{n}t^{n}\Big)\coloneqq\sum_{n=0}^{\infty}a_{n}b_{n}t^{n}.
\]
In our notation for formal power series algebras, we write $t*$ to
indicate that multiplication is the Hadamard product in $t$. For
example, $\mathbb{Q}[[t*,x]]$ is the $\mathbb{Q}$-algebra of formal
power series in the variables $t$ and $x$, where multiplication
is ordinary multiplication in $x$ but is the Hadamard product in
$t$. Thus we have
\begin{align*}
\Big(\sum_{n=0}^{\infty}\sum_{k=0}^{\infty}a_{n,k}x^{k}t^{n}\Big)*\Big(\sum_{n=0}^{\infty}\sum_{k=0}^{\infty}b_{n,k}x^{k}t^{n}\Big) & =\sum_{n=0}^{\infty}\Big(\sum_{k=0}^{\infty}a_{n,k}x^{k}\Big)\Big(\sum_{k=0}^{\infty}b_{n,k}x^{k}\Big)t^{n}\\
 & =\sum_{n=0}^{\infty}\sum_{k=0}^{\infty}\sum_{j=0}^{k}a_{n,j}b_{n,k-j}x^{k}t^{n}
\end{align*}
in $\mathbb{Q}[[t*,x]]$. Note that the identity of $\mathbb{Q}[[t*,x]]$
is $1/(1-t)$.

We write $f^{*\left\langle n\right\rangle }$ to mean the $n$-fold
Hadamard product of $f$, and $f^{*\left\langle -1\right\rangle }$
the inverse of $f$ with respect to Hadamard product. For example,
we have 
\[
(3t+t^{2})^{*\left\langle n\right\rangle }=\underset{n\text{ times}}{\underbrace{(3t+t^{2})*(3t+t^{2})*\cdots*(3t+t^{2})}}=3^{n}t+t^{2}
\]
and
\[
\Big(\frac{1}{1-t}-2xt\Big)^{*\left\langle -1\right\rangle }=\sum_{n=0}^{\infty}(2xt)^{*\left\langle n\right\rangle }=\frac{1}{1-t}+2xt+4x^{2}t+8x^{3}t+\cdots.
\]
We will always use the notations $*$, $*\left\langle n\right\rangle $,
or $*\left\langle -1\right\rangle $ for any expression involving
the Hadamard product; all other expressions should be interpreted
as using ordinary multiplication.

\subsection{Homomorphisms arising from shuffle-compatibility}

Given $\pi\in\mathfrak{S}_{n}$, let us define
\begin{align*}
\Psi_{(\ides,\icomaj)}\colon\mathbf{FQSym} & \rightarrow{\mathcal A}_{(\des,\comaj)}\subseteq\mathbb{Q}[[t*,q,x]],\\
\Psi_{\ipk}\colon\mathrm{\mathbf{FQSym}} & \rightarrow{\mathcal A}_{\pk}\subseteq\mathbb{Q}[[t*,x]],\text{ and}\\
\Psi_{\ilpk}\colon\mathrm{\mathbf{FQSym}} & \rightarrow{\mathcal A}_{\lpk}\subseteq\mathbb{Q}[[t*,x]]
\end{align*}
by
\begin{align*}
\Psi_{(\ides,\icomaj)}(\mathbf{G}_{\pi}) & \coloneqq\begin{cases}
{\displaystyle \frac{t^{\ides(\pi)+1}q^{\icomaj(\pi)}}{\prod_{i=0}^{n}(1-tq^{i})}x^{n}}, & \text{if }n\geq1,\\
1/(1-t), & \text{if }n=0,
\end{cases}\\
\Psi_{\ipk}(\mathbf{G}_{\pi}) & \coloneqq\begin{cases}
{\displaystyle \frac{2^{2\ipk(\pi)+1}t^{\ipk(\pi)+1}(1+t)^{n-2\ipk(\pi)-1}}{(1-t)^{n+1}}x^{n}}, & \text{if }n\geq1,\\
1/(1-t), & \text{if }n=0,\text{ and}
\end{cases}\\
\Psi_{\ilpk}(\mathbf{G}_{\pi}) & \coloneqq\frac{2^{2\ilpk(\pi)}t^{\ilpk(\pi)}(1+t)^{n-2\ilpk(\pi)}}{(1-t)^{n+1}}x^{n}
\end{align*}
and extending linearly.
\begin{thm}
The maps $\Psi_{(\ides,\icomaj)}$, $\Psi_{\ipk}$, and $\Psi_{\ilpk}$
are $\mathbb{Q}$-algebra homomorphisms.
\end{thm}

\begin{proof}
By Theorems 4.5, 4.8, and 4.10 of \cite{Gessel2018}, the descent
statistics $(\des,\comaj)$, $\pk$, and $\lpk$ are all shuffle-compatible
and their homomorphisms
\begin{align*}
\phi_{(\des,\comaj)}\colon\mathrm{QSym} & \rightarrow{\mathcal A}_{(\comaj,\des)}\subseteq\mathbb{Q}[[t*,q,x]],\\
\phi_{\pk}\colon\mathrm{QSym} & \rightarrow{\mathcal A}_{\pk}\subseteq\mathbb{Q}[[t*,x]],\text{ and}\\
\phi_{\lpk}\colon\mathrm{QSym} & \rightarrow{\mathcal A}_{\lpk}\subseteq\mathbb{Q}[[t*,x]]
\end{align*}
(see Theorem \ref{t-gzmain}) are defined by
\begin{align*}
\phi_{(\des,\comaj)}(F_{L}) & \coloneqq\begin{cases}
{\displaystyle \frac{t^{\des(L)+1}q^{\comaj(L)}}{\prod_{i=0}^{n}(1-tq^{i})}x^{n}}, & \text{if }n\geq1,\\
1/(1-t), & \text{if }n=0,
\end{cases}\\
\phi_{\pk}(F_{L}) & \coloneqq\begin{cases}
{\displaystyle \frac{2^{2\pk(L)+1}t^{\pk(L)+1}(1+t)^{n-2\pk(L)-1}}{(1-t)^{n+1}}x^{n}}, & \text{if }n\geq1,\\
1/(1-t), & \text{if }n=0,\text{ and}
\end{cases}\\
\phi_{\lpk}(F_{L}) & \coloneqq\frac{2^{2\lpk(L)}t^{\lpk(L)}(1+t)^{n-2\lpk(L)}}{(1-t)^{n+1}}x^{n},
\end{align*}
where $L\vDash n$. The maps $\Psi_{(\ides,\icomaj)}$, $\Psi_{\ipk}$,
and $\Psi_{\ilpk}$ are simply the result of composing these homomorphisms
$\phi_{(\des,\comaj)}$, $\phi_{\pk}$, and $\phi_{\lpk}$ with the
canonical surjection $\rho$ from $\mathbf{FQSym}$ to $\mathrm{QSym}$\textemdash see
Equation (\ref{e-cansurj}). Since compositions of homomorphisms are
homomorphisms, the result follows.
\end{proof}

\subsection{Further specializations of the generalized cluster method}

We now use the homomorphisms defined in the previous section to produce
further specializations of our generalized cluster method which can
be used to relate the polynomials $A_{\Gamma,n}^{(\ides,\icomaj)}(s,t,q)$,
$A_{\Gamma,n}^{\ides}(s,t)$, $P_{\Gamma,n}^{\ipk}(s,t)$, and $P_{\Gamma,n}^{\ilpk}(s,t)$\textemdash defined
in Section 1.1\textemdash to ``refined cluster polynomials''.
These specializations are similar in spirit to Elizalde's $q$-cluster
method in that they count permutations by occurrences of prescribed
patterns but also keep track of additional statistics.

We begin with $(\ides,\icomaj)$. Given a set $\Gamma\subseteq\mathfrak{S}$,
let
\[
R_{\Gamma,k}^{(\ides,\icomaj)}(s,t,q)\coloneqq\sum_{\substack{\pi\in\mathfrak{S}_{k}}
}t^{\ides(\pi)+1}q^{\icomaj(\pi)}\sum_{c\in C_{\Gamma,\pi}}s^{\mk_{\Gamma}(c)},
\]
which counts $\Gamma$-clusters of length $k$ by the number of marked
occurrences as well as the inverse descent number and inverse comajor
index of the underlying permutation.
\begin{thm}
\label{t-gjcmidesicomaj} Let $\Gamma\subseteq\mathfrak{S}$ be a
set of permutations, each of length at least 2. Then
\begin{multline*}
\quad\sum_{n=0}^{\infty}\frac{A_{\Gamma,n}^{(\ides,\icomaj)}(s,t,q)}{\prod_{i=0}^{n}(1-tq^{i})}x^{n}\\
=\sum_{n=0}^{\infty}\left(\frac{tx}{(1-t)(1-tq)}+\sum_{k=2}^{\infty}R_{\Gamma,k}^{(\ides,\icomaj)}(s-1,t,q)\frac{x^{k}}{\prod_{i=0}^{k}(1-tq^{i})}\right)^{*\left\langle n\right\rangle }\quad.
\end{multline*}
\end{thm}

We note that this formula lives inside the formal power series algebra
$\mathbb{Q}[[s,t*,q,x]]$, although the Hadamard product is only present
on the right-hand side.
\begin{proof}
Take Equation (\ref{e-fqsymgjcm}) from Theorem \ref{t-gjcmfqsym},
and then apply the homomorphism $\Psi_{(\ides,\icomaj)}$ to both sides. Observe that {\allowdisplaybreaks 
\begin{align*}
\Psi_{(\ides,\icomaj)}(\bar{F}_{\Gamma}(s)) & =\sum_{\pi\in\mathfrak{S}_{n}}\Psi_{(\ides,\icomaj)}(\mathbf{G}_{\pi})s^{\occ_{\Gamma}(\pi)}\\
 & =\frac{1}{1-t}+\sum_{n=1}^{\infty}\sum_{\pi\in\mathfrak{S}_{n}}\frac{s^{\occ_{\Gamma}(\pi)}t^{\ides(\pi)+1}q^{\icomaj(\pi)}}{\prod_{i=0}^{n}(1-tq^{i})}x^{n}\\
 & =\sum_{n=0}^{\infty}\sum_{\pi\in\mathfrak{S}_{n}}\frac{A_{\Gamma,n}^{(\ides,\icomaj)}(s,t,q)}{\prod_{i=0}^{n}(1-tq^{i})}x^{n}
\end{align*}
and 
\begin{align*}
\Psi_{(\ides,\icomaj)}(\bar{R}_{\Gamma}(s-1)) & =\sum_{\pi\in\mathfrak{S}}\Psi_{(\ides,\icomaj)}(\mathbf{G}_{\pi})\sum_{c\in C_{\Gamma,\pi}}(s-1)^{\mk_{\Gamma}(c)}\\
 & =\sum_{k=2}^{\infty}\sum_{\pi\in\mathfrak{S}_{k}}\frac{t^{\ides(\pi)+1}q^{\icomaj(\pi)}x^{k}}{\prod_{i=0}^{k}(1-tq^{i})}\sum_{c\in C_{\Gamma,\pi}}(s-1)^{\mk_{\Gamma}(c)}\\
 & =\sum_{k=2}^{\infty}R_{\Gamma,k}^{(\ides,\icomaj)}(s-1,t,q)\frac{x^{k}}{\prod_{i=0}^{k}(1-tq^{i})},
\end{align*}
}and also $\Psi_{(\ides,\icomaj)}(1)=1/(1-t)$ and $\Psi_{(\ides,\icomaj)}(\mathbf{G}_{1})=tx/((1-t)(1-tq))$.
Thus, we have 
\begin{align*}
 & \sum_{n=0}^{\infty}\sum_{\pi\in\mathfrak{S}_{n}}\frac{A_{\Gamma,n}^{(\ides,\icomaj)}(s,t,q)}{\prod_{i=0}^{n}(1-tq^{i})}x^{n}\\
 & \qquad=\left(\frac{1}{1-t}-\frac{tx}{(1-t)(1-tq)}-\sum_{k=2}^{\infty}R_{\Gamma,k}^{(\ides,\icomaj)}(s-1,t,q)\frac{x^{k}}{\prod_{i=0}^{k}(1-tq^{i})}\right)^{*\left\langle -1\right\rangle }\\
 & \qquad=\sum_{n=0}^{\infty}\left(\frac{tx}{(1-t)(1-tq)}+\sum_{k=2}^{\infty}R_{\Gamma,k}^{(\ides,\icomaj)}(s-1,t,q)\frac{x^{k}}{\prod_{i=0}^{k}(1-tq^{i})}\right)^{*\left\langle n\right\rangle }.\qedhere
\end{align*}
\end{proof}
Let us give two remarks before proceeding. First, recall the identity
\[
\comaj(\pi)=n\des(\pi)-\maj(\pi),
\]
which is equivalent to 
\[
\text{\ensuremath{\imaj}}(\pi)=n\text{\ensuremath{\ides}}(\pi)-\text{\ensuremath{\icomaj}}(\pi).
\]
It follows that 
\[
A_{\Gamma,n}^{(\ides,\imaj)}(s,t,q)=q^{-n}A_{\Gamma,n}^{(\ides,\icomaj)}(s,tq^{n},q^{-1}),
\]
so we can compute the polynomials $A_{\Gamma,n}^{(\ides,\imaj)}(s,t,q)$
from the $A_{\Gamma,n}^{(\ides,\icomaj)}(s,t,q)$. In other words,
having a formula for the polynomials $A_{\Gamma,n}^{(\ides,\icomaj)}(s,t,q)$
is equivalent to having one for the $A_{\Gamma,n}^{(\ides,\imaj)}(s,t,q)$.

Furthermore, in using Theorem \ref{t-gjcmidesicomaj} to compute the
polynomial $A_{\Gamma,j}^{(\ides,\icomaj)}(s,t,q)$, one only needs
to sum from $n=0$ to $n=j$ on the right-hand side. This is because,
by the definition of Hadamard product in $t$, the coefficient of
$x^{j}$ in 
\[
\left(\frac{tx}{(1-t)(1-tq)}+\sum_{k=2}^{\infty}R_{\Gamma,k}^{(\ides,\icomaj)}(s-1,t,q)\frac{x^{k}}{\prod_{i=0}^{k}(1-tq^{i})}\right)^{*\left\langle n\right\rangle }
\]
is zero unless $n\leq j$. The same is true for the other formulas
in this section.

We now specialize Theorem \ref{t-gjcmidesicomaj} to an analogous
result solely for the inverse descent number. Let 
\[
R_{\Gamma,k}^{\ides}(s,t)\coloneqq\sum_{\substack{\pi\in\mathfrak{S}_{k}}
}t^{\ides(\pi)+1}\sum_{c\in C_{\Gamma,\pi}}s^{\mk_{\Gamma}(c)}
\]
be the refined cluster polynomial for $\ides$.
\begin{thm}
\label{t-gjcmides} Let $\Gamma\subseteq\mathfrak{S}$ be a set of
permutations, each of length at least 2. Then
\begin{align*}
\sum_{n=0}^{\infty}\frac{A_{\Gamma,n}^{\ides}(s,t)}{(1-t)^{n+1}}x^{n} & =\sum_{n=0}^{\infty}\left(\frac{tx}{(1-t)^{2}}+\frac{1}{1-t}\sum_{k=2}^{\infty}R_{\Gamma,k}^{\ides}(s-1,t)z^{k}\right)^{*\left\langle n\right\rangle }
\end{align*}
where $z=x/(1-t)$.
\end{thm}

\begin{proof}
This follows immediately from setting $q=1$ in Theorem \ref{t-gjcmidesicomaj}
and simplifying.
\end{proof}
For the inverse peak statistics $\ipk$ and $\ilpk$, let us define
\[
R_{\Gamma,k}^{\ipk}(s,t)\coloneqq\sum_{\substack{\pi\in\mathfrak{S}_{k}}
}t^{\ipk(\pi)+1}\sum_{c\in C_{\Gamma,\pi}}s^{\mk_{\Gamma}(c)}\;\: \text{and}\;\: R_{\Gamma,k}^{\ilpk}(s,t)\coloneqq\sum_{\substack{\pi\in\mathfrak{S}_{k}}
}t^{\ilpk(\pi)}\sum_{c\in C_{\Gamma,\pi}}s^{\mk_{\Gamma}(c)}.
\]
Then the following two theorems can be proven in the same way as Theorem
\ref{t-gjcmidesicomaj}, but
using the homomorphisms $\Psi_{\ipk}$ and $\Psi_{\ilpk}$. We outline
the steps for Theorem \ref{t-gjcmipk} but omit the proof of Theorem
\ref{t-gjcmilpk}.
\begin{thm}
\label{t-gjcmipk} Let $\Gamma\subseteq\mathfrak{S}$ be a set of
permutations, each of length at least 2. Then
\begin{equation*}
\quad\frac{1}{1-t}+\frac{1+t}{2(1-t)}\sum_{n=1}^{\infty}P_{\Gamma,n}^{\ipk}(s,u)z^{n}
=\sum_{n=0}^{\infty}\left(\frac{2tx}{(1-t)^{2}}+\frac{1+t}{2(1-t)}\sum_{k=2}^{\infty}R_{\Gamma,k}^{\ipk}(s-1,u)z^{k}\right)^{*\left\langle n\right\rangle }\quad
\end{equation*}
where $u=4t/(1+t)^{2}$ and $z=(1+t)x/(1-t)$.
\end{thm}

\begin{proof}
We shall apply $\Psi_{\ipk}$ to both sides of (\ref{e-fqsymgjcm}). Observe that {\allowdisplaybreaks
\begin{align*}
\Psi_{\ipk}(\bar{F}_{\Gamma}(s)) & =\frac{1}{1-t}+\sum_{n=1}^{\infty}\sum_{\pi\in\mathfrak{S}_{n}}\frac{2^{2\ipk(\pi)+1}t^{\ipk(\pi)+1}(1+t)^{n-2\ipk(\pi)-1}s^{\occ_{\Gamma}(\pi)}}{(1-t)^{n+1}}x^{n}\\
 & =\frac{1}{1-t}+\frac{1}{2}\sum_{n=1}^{\infty}\left(\frac{1+t}{1-t}\right)^{n+1}P_{\Gamma,n}^{\ipk}\hspace{-2pt}\left(s,\frac{4t}{(1+t)^{2}}\right)x^{n}\\
 & =\frac{1}{1-t}+\frac{1+t}{2(1-t)}\sum_{n=1}^{\infty}P_{\Gamma,n}^{\ipk}(s,u)z^{n}
\end{align*}
and
\begin{align*}
\Psi_{\ipk}(\bar{R}_{\Gamma}(s-1))
 &= \sum_{k=2}^{\infty}\sum_{\pi\in\mathfrak{S}_{k}}\frac{2^{2\ipk(\pi)+1}t^{\ipk(\pi)+1}(1+t)^{k-2\ipk(\pi)-1}x^{k}}{(1-t)^{k+1}}\sum_{c\in C_{\Gamma,\pi}}(s-1)^{\mk_{\Gamma}(c)}\\
 &=\sum_{k=2}^{\infty}R_{\Gamma,k}^{\ipk}\hspace{-2pt}\left(s-1,\frac{4t}{(1+t)^{2}}\right)\frac{(1+t)^{k+1}x^{k}}{2(1-t)^{k+1}}\\
 &=\frac{1+t}{2(1-t)}\sum_{k=2}^{\infty}R_{\Gamma,k}^{\ipk}(s-1,u)z^{k}.
\end{align*}
}Also, we have $\Psi_{\ipk}(1)=1/(1-t)$ and $\Psi_{\ipk}(\mathbf{G}_{1})=2tx/(1-t)^{2}$.
Hence, we obtain
\begin{align*}
 \frac{1}{1-t}+\frac{1+t}{2(1-t)}\sum_{n=1}^{\infty}P_{\Gamma,n}^{\ipk}(s,u)z^{n}
 &=\left(\frac{1}{1-t}-\frac{2tx}{(1-t)^{2}}-\frac{1+t}{2(1-t)}\sum_{k=2}^{\infty}R_{\Gamma,k}^{\ipk}(s-1,u)z^{k}\right)^{*\left\langle -1\right\rangle }\\
 &=\sum_{n=0}^{\infty}\left(\frac{2tx}{(1-t)^{2}}+\frac{1+t}{2(1-t)}\sum_{k=2}^{\infty}R_{\Gamma,k}^{\ipk}(s-1,u)z^{k}\right)^{*\left\langle n\right\rangle }.\qedhere
\end{align*}
\end{proof}
\begin{thm}
\label{t-gjcmilpk} Let $\Gamma\subseteq\mathfrak{S}$ be a set of
permutations, each of length at least 2. Then
\begin{align*}
\frac{1}{1-t}\sum_{n=0}^{\infty}P_{\Gamma,n}^{\ilpk}(s,u)z^{n} & =\sum_{n=0}^{\infty}\left(\frac{z}{1-t}+\frac{1}{1-t}\sum_{k=2}^{\infty}R_{\Gamma,k}^{\ilpk}(s-1,u)z^{k}\right)^{*\left\langle n\right\rangle }
\end{align*}
where $u=4t/(1+t)^{2}$ and $z=(1+t)x/(1-t)$.
\end{thm}

\section{Monotone patterns \texorpdfstring{$12\cdots m$}{12...m} and \texorpdfstring{$m\cdots 21$}{m...21}}

\subsection{Cluster generating functions for monotone patterns}

In this section, we will study the polynomials $A_{\sigma,n}^{(\ides,\imaj)}(t,q)$,
$A_{\sigma,n}^{\ides}(s,t,q)$, $A_{\sigma,n}^{\ides}(t,q)$, $P_{\sigma,n}^{\ipk}(s,t)$,
$P_{\sigma,n}^{\ipk}(t)$, $P_{\sigma,n}^{\ilpk}(s,t)$, and $P_{\sigma,n}^{\ilpk}(t)$
for $\sigma=12\cdots m$ and $\sigma=m\cdots21$. Our formulas will
mostly be for the pattern $\sigma=12\cdots m$, but we can use these
formulas along with Proposition \ref{p-rcpoly} to compute most of
these polynomials for $\sigma=m\cdots21$ as well. 

We begin with a lemma establishing closed-form generating functions
for refined $12\cdots m$-cluster polynomials, which we need in order
to apply our results from Section 3.5. Note that, in general, there
is no straightforward way to count clusters by inverse statistics.
As a matter of fact, the simpler problem of counting clusters (without
keeping track of any statistic) is equivalent to counting linear extensions
of a certain poset \cite{Elizalde2012}, which is itself difficult
in general. Yet, counting $\sigma$-clusters by our inverse statistics
is essentially trivial when $\sigma$ is a monotone pattern.
\begin{lem}
\label{l-incR} For all $m\geq2$, we have 
\begin{align*}
\sum_{k=2}^{\infty}R_{12\cdots m,k}^{(\ides,\icomaj)}(s,t,q)x^{k} =\sum_{k=2}^{\infty}R_{12\cdots m,k}^{\ides}(s,t)x^{k}
 =\sum_{k=2}^{\infty}R_{12\cdots m,k}^{\ipk}(s,t)x^{k}
 =\frac{stx^{m}}{1-s\sum_{l=1}^{m-1}x^{l}}
\end{align*}
and 
\[
\sum_{k=2}^{\infty}R_{12\cdots m,k}^{\ilpk}(s,t)x^{k}=\frac{sx^{m}}{1-s\sum_{l=1}^{m-1}x^{l}}.
\]
\end{lem}

\begin{proof}
It is easy to see that there exists a $12\cdots m$-cluster on $\pi$
if and only if $\pi$ is itself of the form $12\cdots n$ where $n\geq m$,
and that the overlap set of $12\cdots m$ is given by $O_{12\cdots m}=\{1,2,\dots,m-1\}$.
Hence, we can uniquely generate $12\cdots m$-clusters by first taking
the permutation $12\cdots m$, and then repeatedly appending the next
$l$ largest integers (for any $1\leq l\leq m-1$) in increasing order\textemdash each
iteration creates an additional marked occurrence of $12\cdots m$.
Figure 1 provides an illustration for the case $m=4$. Thus, we have
the formula 
\begin{align}
\sum_{k=2}^{\infty}\sum_{\pi\in\mathfrak{S}_{k}}\sum_{c\in C_{12\cdots m,\pi}}s^{\mk_{12\cdots m}(c)}x^{k} & =\frac{sx^{m}}{1-s(x+x^{2}+\cdots+x^{m-1})}\nonumber \\
 & =\frac{sx^{m}}{1-s\sum_{l=1}^{m-1}x^{l}}\label{e-rcpinc}
\end{align}
(see also \cite[p. 356]{Elizalde2012}). 
\begin{figure}
\begin{center}
\begin{tikzpicture}


\node at (0,0){
\begin{minipage}{\textwidth}
\begin{alignat*}{3} 
1\;2\;3\;4\quad\mapsto\quad & 1\;2\;3\;4\;5\;6\;7 & \quad\mapsto\quad & 1\;2\;3\;4\;5\;6\;7\;8\;9\;10 & \quad\mapsto\cdots\\[10pt]  &  & \quad\mapsto\quad & 1\;2\;3\;4\;5\;6\;7\;8\;9 & \quad\mapsto\cdots\\[10pt]  &  & \quad\mapsto\quad & 1\;2\;3\;4\;5\;6\;7\;8 & \quad\mapsto\cdots\\[10pt] \quad\mapsto\quad & 1\;2\;3\;4\;5\;6 & \quad\mapsto\quad & 1\;2\;3\;4\;5\;6\;7\;8\;9 & \quad\mapsto\cdots\\[10pt]  &  & \quad\mapsto\quad & 1\;2\;3\;4\;5\;6\;7\;8 & \quad\mapsto\cdots\\[10pt]  &  & \quad\mapsto\quad & 1\;2\;3\;4\;5\;6\;7 & \quad\mapsto\cdots\\[10pt] \quad\mapsto\quad & 1\;2\;3\;4\;5 & \quad\mapsto\quad & 1\;2\;3\;4\;5\;6\;7\;8 & \quad\mapsto\cdots\\[10pt]  &  & \quad\mapsto\quad & 1\;2\;3\;4\;5\;6\;7 & \quad\mapsto\cdots\\[10pt]  &  & \quad\mapsto\quad & 1\;2\;3\;4\;5\;6 & \quad\mapsto\cdots 
\end{alignat*}
\end{minipage}
};

\draw[red] (-4.75,3.61) ellipse (19bp and 11bp);
\draw[red] (-2.23,3.61) ellipse (19bp and 11bp);
\draw[red] (-1.22,3.61) ellipse (19bp and 11bp);
\draw[red] (1.17,3.61) ellipse (19bp and 11bp);
\draw[red] (2.15,3.61) ellipse (19bp and 11bp);
\draw[red] (3.22,3.61) ellipse (22bp and 11bp);

\draw[red] (1.17,2.64) ellipse (19bp and 11bp);
\draw[red] (2.15,2.64) ellipse (19bp and 11bp);
\draw[red] (2.8,2.64) ellipse (19bp and 11bp);

\draw[red] (1.15,1.68) ellipse (19bp and 11bp);
\draw[red] (2.15,1.68) ellipse (19bp and 11bp);
\draw[red] (2.5,1.68) ellipse (19bp and 11bp);

\draw[red] (-2.23,0.71) ellipse (19bp and 11bp);
\draw[red] (-1.58,0.71) ellipse (19bp and 11bp);
\draw[red] (1.17,0.71) ellipse (19bp and 11bp);
\draw[red] (1.82,0.71) ellipse (19bp and 11bp);
\draw[red] (2.8,0.71) ellipse (19bp and 11bp);

\draw[red] (1.15,-0.25) ellipse (19bp and 11bp);
\draw[red] (1.82,-0.25) ellipse (19bp and 11bp);
\draw[red] (2.5,-0.25) ellipse (19bp and 11bp);

\draw[red] (1.17,-1.23) ellipse (19bp and 11bp);
\draw[red] (1.82,-1.23) ellipse (19bp and 11bp);
\draw[red] (2.18,-1.23) ellipse (19bp and 11bp);

\draw[red] (-2.23,-2.19) ellipse (19bp and 11bp);
\draw[red] (-1.91,-2.19) ellipse (19bp and 11bp);
\draw[red] (1.15,-2.19) ellipse (19bp and 11bp);
\draw[red] (1.49,-2.19) ellipse (19bp and 11bp);
\draw[red] (2.5,-2.19) ellipse (19bp and 11bp);

\draw[red] (1.17,-3.15) ellipse (19bp and 11bp);
\draw[red] (1.49,-3.15) ellipse (19bp and 11bp);
\draw[red] (2.18,-3.15) ellipse (19bp and 11bp);

\draw[red] (1.17,-4.12) ellipse (19bp and 11bp);
\draw[red] (1.49,-4.12) ellipse (19bp and 11bp);
\draw[red] (1.83,-4.12) ellipse (19bp and 11bp);

\end{tikzpicture}
\end{center}

\caption{$1234$-clusters}
\end{figure}
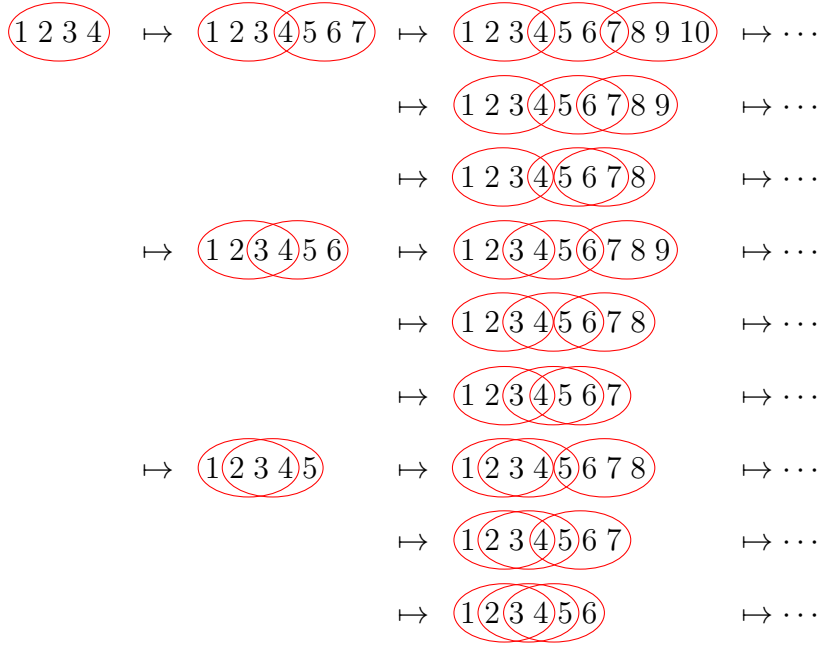

Moreover, since $12\cdots m$-clusters are themselves monotone increasing,
their inverses are also monotone increasing and therefore have no
descents, peaks, or left peaks. Using (\ref{e-rcpinc}), it follows
that {\allowdisplaybreaks
\begin{align*}
\sum_{k=2}^{\infty}R_{12\cdots m,k}^{(\ides,\icomaj)}(s,t,q)x^{k} & =\sum_{k=2}^{\infty}\sum_{\pi\in\mathfrak{S}_{k}}t^{\ides(\pi)+1}q^{\icomaj(\pi)}\sum_{c\in C_{12\cdots m,\pi}}s^{\mk_{12\cdots m}(c)}x^{k}\\
 & =t\sum_{k=2}^{\infty}\sum_{\pi\in\mathfrak{S}_{k}}\sum_{c\in C_{12\cdots m,\pi}}s^{\mk_{12\cdots m}(c)}x^{k} \\ 
 & =\frac{stx^{m}}{1-s\sum_{l=1}^{m-1}x^{l}};
\end{align*}
our formulas for $\sum_{k=2}^{\infty}R_{12\cdots m,k}^{\ipk}(s,t)x^{k}$
and $\sum_{k=2}^{\infty}R_{12\cdots m,k}^{\ilpk}(s,t)x^{k}$ are obtained
in the same way. Lastly, since our formula for $\sum_{k=2}^{\infty}R_{12\cdots m,k}^{(\ides,\icomaj)}(s,t,q)x^{k}$
does not depend on $q$, we have the same formula for $\sum_{k=2}^{\infty}R_{12\cdots m,k}^{\ides}(s,t)x^{k}$.}
\end{proof}

\subsection{Monotone patterns, inverse descent number, and inverse major index}

We will now derive generating function formulas for $A_{12\cdots m,n}^{(\ides,\imaj)}(t,q)$,
$A_{12\cdots m,n}^{\ides}(s,t)$, and $A_{12\cdots m,n}^{\ides}(t)$.
\begin{thm}
\label{t-incidesimaj} Let $m\geq2$. We have \leqnomode
\begin{multline*}
\qquad\tag{{a}}\sum_{n=0}^{\infty}\frac{A_{12\cdots m,n}^{(\ides,\imaj)}(t,q)}{\prod_{i=0}^{n}(1-tq^{i})}x^{n}\\
=\sum_{n=0}^{\infty}\left(\frac{tx}{(1-t)(1-tq)}-\sum_{j=1}^{\infty}\left(\frac{tx^{jm}}{\prod_{i=0}^{jm}(1-tq^{i})}-\frac{tx^{jm+1}}{\prod_{i=0}^{jm+1}(1-tq^{i})}\right)\right)^{*\left\langle n\right\rangle }\qquad
\end{multline*}
and
\begin{multline*}
\tag{{b}}\qquad\sum_{n=0}^{\infty}\frac{A_{12\cdots m,n}^{(\ides,\imaj)}(t,q)}{\prod_{i=0}^{n}(1-tq^{i})}x^{n}\\
=1+\sum_{k=1}^{\infty}\left[\sum_{j=0}^{\infty}\left({k+jm-1 \choose k-1}_{\!\!q}x^{jm}-{k+jm \choose k-1}_{\!\!q}x^{jm+1}\right)\right]^{-1}t^{k}.\qquad
\end{multline*}
\end{thm}

\begin{proof}
In light of Proposition \ref{p-rcpoly} (b) and the fact that $12\cdots m$
is invariant under reverse-complementation, it suffices to prove these
formulas with the polynomial $A_{12\cdots m,n}^{(\ides,\imaj)}(t,q)$
replaced by $A_{12\cdots m,n}^{(\ides,\icomaj)}(t,q)$.

We first apply Theorem \ref{t-gjcmidesicomaj} to $\Gamma=\{12\cdots m\}$
and set $s=0$ to obtain
\begin{multline}
\qquad\sum_{n=0}^{\infty}\frac{A_{12\cdots m,n}^{(\ides,\icomaj)}(t,q)}{\prod_{i=0}^{n}(1-tq^{i})}x^{n}\\
=\sum_{n=0}^{\infty}\left(\frac{tx}{(1-t)(1-tq)}+\sum_{k=2}^{\infty}R_{12\cdots m,k}^{(\ides,\icomaj)}(-1,t,q)\frac{x^{k}}{\prod_{i=0}^{k}(1-tq^{i})}\right)^{*\left\langle n\right\rangle }.\qquad\label{e-incidesicomaj}
\end{multline}
By Lemma \ref{l-incR}, we have 
\begin{align*}
\sum_{k=2}^{\infty}R_{12\cdots m,k}^{(\ides,\icomaj)}(-1,t,q)x^{k} & =\frac{-tx^{m}}{1+x+x^{2}+\cdots+x^{m-1}}\\
 & =-\frac{tx^{m}(1-x)}{1-x^{m}}\\
 & =-\sum_{j=1}^{\infty}(tx^{jm}-tx^{jm+1})
\end{align*}
and thus
\begin{equation}
\sum_{k=2}^{\infty}R_{12\cdots m,k}^{(\ides,\icomaj)}(-1,t,q)\frac{x^{k}}{\prod_{i=0}^{k}(1-tq^{i})}
=-\sum_{j=1}^{\infty}\left(\frac{tx^{jm}}{\prod_{i=0}^{jm}(1-tq^{i})}-\frac{tx^{jm+1}}{\prod_{i=0}^{jm+1}(1-tq^{i})}\right).\label{e-incidesicomajR}
\end{equation}
Combining (\ref{e-incidesicomaj}) with (\ref{e-incidesicomajR})
yields part (a).

Now we prove part (b). Here we begin with (\ref{e-incidesicomajR}),
and use the identity
\[
\frac{1}{\prod_{i=0}^{n}(1-tq^{i})}=\sum_{k=0}^{\infty}{n+k \choose k}_{\!\!q}t^{k}
\]
\cite[p. 68]{Stanley2011} to arrive at
\begin{align*}
 & \sum_{n=0}^{\infty}\frac{A_{12\cdots m,n}^{(\ides,\icomaj)}(t,q)}{\prod_{i=0}^{n}(1-tq^{i})}x^{n} \\
 & \qquad = \sum_{n=0}^{\infty}\left(\sum_{k=0}^{\infty}\left({k+1 \choose k}_{\!\!q}x-\sum_{j=1}^{\infty}{k+jm \choose k}_{\!\!q}x^{jm}+\sum_{j=1}^{\infty}{k+jm+1 \choose k}_{\!\!q}x^{jm+1}\right)t^{k+1}\right)^{\hspace{-2.5pt}*\left\langle n\right\rangle }.
\end{align*}
A sequence of algebraic manipulations yields {\allowdisplaybreaks
\begin{align*}
 & \sum_{n=0}^{\infty}\frac{A_{12\cdots m,n}^{(\ides,\icomaj)}(t,q)}{\prod_{i=0}^{n}(1-tq^{i})}x^{n}\\
 & \quad =\frac{1}{1-t} +\sum_{k=0}^{\infty}\sum_{n=1}^{\infty}\left({k+1 \choose k}_{\!\!q}x-\sum_{j=1}^{\infty}{k+jm \choose k}_{\!\!q}x^{jm}+\sum_{j=1}^{\infty}{k+jm+1 \choose k}_{\!\!q}x^{jm+1}\right)^{\hspace{-2.5pt}n}t^{k+1}\\
 & \quad =1+\sum_{k=0}^{\infty}\sum_{n=0}^{\infty}\left({k+1 \choose k}_{\!\!q}x-\sum_{j=1}^{\infty}{k+jm \choose k}_{\!\!q}x^{jm}+\sum_{j=1}^{\infty}{k+jm+1 \choose k}_{\!\!q}x^{jm+1}\right)^{\hspace{-2.5pt}n}t^{k+1}\\
 & \quad =1+\sum_{k=0}^{\infty}\left(1-{k+1 \choose k}_{\!\!q}x+\sum_{j=1}^{\infty}{k+jm \choose k}_{\!\!q}x^{jm}-\sum_{j=1}^{\infty}{k+jm+1 \choose k}_{\!\!q}x^{jm+1}\right)^{\hspace{-2.7pt}-1}t^{k+1}\\
 & \quad =1+\sum_{k=0}^{\infty}\left(\sum_{j=0}^{\infty}{k+jm \choose k}_{\!\!q}x^{jm}-\sum_{j=0}^{\infty}{k+jm+1 \choose k}_{\!\!q}x^{jm+1}\right)^{\hspace{-2.7pt}-1}t^{k+1}\\
 & \quad =1+\sum_{k=1}^{\infty}\left[\sum_{j=0}^{\infty}\left({k+jm-1 \choose k-1}_{\!\!q}x^{jm}-{k+jm \choose k-1}_{\!\!q}x^{jm+1}\right)\right]^{-1}t^{k},
\end{align*}
thus completing the proof.}
\end{proof}
Let $A_{n}(t,q)\coloneqq\sum_{\pi\in\mathfrak{S}_{n}}t^{\des(\pi)+1}q^{\maj(\pi)}$
for $n\geq1$ and $A_{0}(t,q)\coloneqq1$; these are called $q$-\textit{Eulerian
polynomials} and encode the joint distribution of $\des$ and $\maj$
over $\mathfrak{S}_{n}$. Observe that
\[
\lim_{m\rightarrow\infty}A_{12\cdots m,n}^{(\ides,\imaj)}(t,q)=\sum_{\pi\in\mathfrak{S}_{n}}t^{\ides(\pi)+1}q^{\imaj(\pi)}=\sum_{\pi\in\mathfrak{S}_{n}}t^{\des(\pi)+1}q^{\maj(\pi)}=A_{n}(t,q);
\]
we can exploit this limit to recover from Theorem \ref{t-incidesimaj}
a classical identity involving $q$-Eulerian polynomials. By taking
the limit as $m\rightarrow\infty$ of both sides of Theorem \ref{t-incidesimaj}
(b), we obtain
\begin{align*}
\sum_{n=0}^{\infty}\frac{A_{n}(t,q)}{\prod_{i=0}^{n}(1-tq^{i})}x^{n} & =1+\sum_{k=1}^{\infty}(1-[k]_{q}x)^{-1}t^{k}\\
 & =1+\sum_{k=1}^{\infty}\sum_{n=0}^{\infty}[k]_{q}^{n}x^{n}t^{k}\\
 & =1+\sum_{n=0}^{\infty}\sum_{k=1}^{\infty}[k]_{q}^{n}t^{k}x^{n},
\end{align*}
and extracting coefficients of $x^{n}$ yields the famous \textit{Carlitz
identity} \cite[Corollary 6.1]{Petersen2015}
\[
\frac{A_{n}(t,q)}{\prod_{i=0}^{n}(1-tq^{i})}=\sum_{k=1}^{\infty}[k]_{q}^{n}t^{k}.
\]

Next, we have the following formulas for the polynomials $A_{12\cdots m,n}^{\ides}(s,t)$
and $A_{12\cdots m,n}^{\ides}(t)$.
\begin{thm}
\label{t-incides} Let $m\geq2$. Then \leqnomode
\begin{align*}
\tag{{a}}\sum_{n=0}^{\infty}\frac{A_{12\cdots m,n}^{\ides}(s,t)}{(1-t)^{n+1}}x^{n} & =\sum_{n=0}^{\infty}\Bigg(\frac{tx}{(1-t)^{2}}+\frac{(s-1)tz^{m}}{(1-t)(1-(s-1)\sum_{l=1}^{m-1}z^{l})}\Bigg)^{*\left\langle n\right\rangle },
\end{align*}
\begin{align*}
\tag{{b}}\sum_{n=0}^{\infty}\frac{A_{12\cdots m,n}^{\ides}(t)}{(1-t)^{n+1}}x^{n} & =\sum_{n=0}^{\infty}\Bigg(\frac{tz(1-z^{m-1})}{(1-t)(1-z^{m})}\Bigg)^{*\left\langle n\right\rangle },
\end{align*}
and
\begin{align*}
\tag{{c}}\sum_{n=0}^{\infty}\frac{A_{12\cdots m,n}^{\ides}(t)}{(1-t)^{n+1}}x^{n}=
1+\sum_{k=1}^{\infty}\left[\sum_{j=0}^{\infty}\left({k+jm-1 \choose k-1}x^{jm}-{k+jm \choose k-1}x^{jm+1}\right)\right]^{-1}t^{k},
\end{align*}
where $z=x/(1-t)$.
\end{thm}

\begin{proof}
Part (a) follows immediately from Theorem \ref{t-gjcmides} and Lemma
\ref{l-incR}, and part (c) is obtained from substituting $q=1$ into
Theorem \ref{t-incidesimaj} (b). Then taking $s=0$ in part (a),
we have
\begin{align*}
\sum_{n=0}^{\infty}\frac{A_{12\cdots m,n}^{\ides}(t)}{(1-t)^{n+1}}x^{n} & =\sum_{n=0}^{\infty}\Bigg(\frac{tx}{(1-t)^{2}}-\frac{tz^{m}}{(1-t)(1+z+z^{2}+\cdots+z^{m-1})}\Bigg)^{*\left\langle n\right\rangle }\\
 & =\sum_{n=0}^{\infty}\Bigg(\frac{tz}{1-t}-\frac{tz^{m}(1-z)}{(1-t)(1-z^{m})}\Bigg)^{*\left\langle n\right\rangle }\\
 & =\sum_{n=0}^{\infty}\Bigg(\frac{tz(1-z^{m-1})}{(1-t)(1-z^{m})}\Bigg)^{*\left\langle n\right\rangle }
\end{align*}
which proves part (b).
\end{proof}
We use Theorem \ref{t-incides} to compute the first ten polynomials
$A_{12\cdots m,n}^{\ides}(t)$ for $m=3$ and $m=4$, which are displayed
in Tables 1\textendash 2. (By Proposition \ref{p-rcpoly} (d), the
polynomials $A_{m\cdots21,n}^{\ides}(t)$ are the same as the $A_{12\cdots m,n}^{\ides}(t)$
but with the order of their coefficients reversed.)

\renewcommand{\arraystretch}{1.2}

\begin{table}[H]
\centering{}%
\begin{tabular}{c|c}
$n$ & $A_{123,n}^{\ides}(t)$\tabularnewline
\hline 
$0$ & 1\tabularnewline
$1$ & $t$\tabularnewline
$2$ & $t+t^{2}$\tabularnewline
$3$ & $4t^{2}+t^{3}$\tabularnewline
$4$ & $5t^{2}+11t^{3}+t^{4}$\tabularnewline
$5$ & $4t^{2}+39t^{3}+26t^{4}+t^{5}$\tabularnewline
$6$ & $5t^{2}+91t^{3}+195t^{4}+57t^{5}+t^{6}$\tabularnewline
$7$ & $4t^{2}+193t^{3}+904t^{4}+795t^{5}+120t^{6}+t^{7}$\tabularnewline
$8$ & $5t^{2}+396t^{3}+3420t^{4}+6400t^{5}+2889t^{6}+247t^{7}+t^{8}$\tabularnewline
$9$ & $4t^{2}+761t^{3}+11610t^{4}+39275t^{5}+37450t^{6}+9774t^{7}+502t^{8}+t^{9}$\tabularnewline
\end{tabular}\vspace{5bp}
\caption{Distribution of $\protect\ides$ over $\mathfrak{S}_{n}(123)$}
\end{table}

\begin{table}
\centering{}%
\begin{tabular}{c|c}
$n$ & $A_{1234,n}^{\ides}(t)$\tabularnewline
\hline 
$0$ & 1\tabularnewline
$1$ & $t$\tabularnewline
$2$ & $t+t^{2}$\tabularnewline
$3$ & $t+4t^{2}+t^{3}$\tabularnewline
$4$ & $11t^{2}+11t^{3}+t^{4}$\tabularnewline
$5$ & $18t^{2}+66t^{3}+26t^{4}+t^{5}$\tabularnewline
$6$ & $28t^{2}+254t^{3}+302t^{4}+57t^{5}+t^{6}$\tabularnewline
$7$ & $40t^{2}+814t^{3}+2160t^{4}+1191t^{5}+120t^{6}+t^{7}$\tabularnewline
$8$ & $64t^{2}+2358t^{3}+12030t^{4}+14340t^{5}+4293t^{6}+247t^{7}+t^{8}$\tabularnewline
$9$ & $96t^{2}+6538t^{3}+57804t^{4}+127250t^{5}+82102t^{6}+14608t^{7}+502t^{8}+t^{9}$\tabularnewline
\end{tabular}\vspace{5bp}
\caption{Distribution of $\protect\ides$ over $\mathfrak{S}_{n}(1234)$}
\end{table}

Curiously, beginning with $n=3$, the quadratic coefficient of $A_{123,n}^{\ides}(t)$
alternates between 4 and 5. We state this observation in the following
proposition.
\begin{prop}
\label{p-123ides} Let $n\geq3$. The number of permutations $\pi$
in $\mathfrak{S}_{n}(123)$ with $\ides(\pi)=1$ is 4 if $n$ is odd,
and is 5 if $n$ is even.
\end{prop}

Proposition \ref{p-123ides} can be proven using Theorem \ref{t-incides},
but we shall instead sketch a combinatorial proof, which is more enlightening.
Our proof relies on the notion of reading sequences of permutations.
Given a permutation $\pi\in\mathfrak{S}_{n}$, we read the letters
$1,2,\dots,n$ in $\pi$ from left-to-right in order, going back to
the beginning of $\pi$ when necessary; this process realizes $\pi$
as a shuffle of \textit{reading sequences} \cite[p. 37]{Stanley2011}.
For example, take $\pi=748361259$; then the reading sequences of
$\pi$ are $12$, $3$, $45$, $6$, and $789$. It is easy to see
that the lengths of reading sequences of $\pi$ are precisely the
lengths of the increasing runs of $\pi^{-1}$. Thus, the inverse of
a permutation $\pi$ avoids $12\cdots m$ if and only if every reading
sequence of $\pi$ has length less than $m$.
\begin{proof}
The number of permutations $\pi$ in $\mathfrak{S}_{n}(123)$ with
$\ides(\pi)=1$ is equal to the number of permutations $\pi$ with
$\des(\pi)=1$ whose inverse $\pi^{-1}$ is in $\mathfrak{S}_{n}(123)$,
so it suffices to prove the result for the latter family of permutations.
More specifically, we claim that if $n$ is odd, then the permutations
$\pi$ in $\mathfrak{S}_{n}$ with $\des(\pi)=1$ whose inverse $\pi^{-1}$
is in $\mathfrak{S}_{n}(123)$ are
\begin{itemize}
\item $13\cdots n24\cdots(n-1)$,
\item $24\cdots(n-1)13\cdots n$,
\item $24\cdots(n-1)n13\cdots(n-2)$, and
\item $35\cdots n124\cdots(n-1)$.
\end{itemize}
Moreover, if $n$ is even, then the permutations $\pi$ in $\mathfrak{S}_{n}$
with $\des(\pi)=1$ whose inverse $\pi^{-1}$ is in $\mathfrak{S}_{n}(123)$
are
\begin{itemize}
\item $13\cdots(n-1)24\cdots n$
\item $13\cdots(n-1)n24\cdots(n-2)$,
\item $24\cdots n13\cdots(n-1)$,
\item $35\cdots(n-1)124\cdots n$, and
\item $35\cdots(n-1)n124\cdots(n-2)$.
\end{itemize}
Clearly, all of these permutations have exactly one descent and the
lengths of their reading sequences are all less than 3. A careful
case analysis shows that these are the only permutations in $\mathfrak{S}_{n}$
with these properties; we omit the details.
\end{proof}

\subsection{Monotone patterns and inverse peak number}

Next, we proceed to the polynomials $P_{12\cdots m,n}^{\ipk}(s,t)$
and $P_{12\cdots m,n}^{\ipk}(t)$, which are equal to the polynomials
$P_{m\cdots21,n}^{\ipk}(s,t)$ and $P_{m\cdots21,n}^{\ipk}(t)$ by
Proposition \ref{p-rcpoly} (e).
\begin{thm}
\label{t-monoipk} Let $m\geq2$. We have \leqnomode 
\begin{multline*}
\tag{{a}}\qquad\frac{1}{1-t}+\frac{1+t}{2(1-t)}\sum_{n=1}^{\infty}P_{12\cdots m,n}^{\ipk}(s,u)z^{n}\\
=\sum_{n=0}^{\infty}\Bigg(\frac{2tx}{(1-t)^{2}}+\frac{2t(s-1)z^{m}}{(1-t^{2})(1-(s-1)\sum_{l=1}^{m-1}z^{l})}\Bigg)^{*\left\langle n\right\rangle },\qquad
\end{multline*}
\begin{alignat*}{1}
\tag{{b}}\frac{1}{1-t}+\frac{1+t}{2(1-t)}\sum_{n=1}^{\infty}P_{12\cdots m,n}^{\ipk}(u)z^{n} & =\sum_{n=0}^{\infty}\left(\frac{2tz(1-z^{m-1})}{(1-t^{2})(1-z^{m})}\right)^{*\left\langle n\right\rangle },
\end{alignat*}
and 
\begin{multline*}
\tag{{c}}\qquad\frac{1}{1-t}+\frac{1+t}{2(1-t)}\sum_{n=1}^{\infty}P_{12\cdots m,n}^{\ipk}(u)z^{n}\\
=1+\sum_{k=1}^{\infty}\left[1-2kx+\sum_{j=1}^{\infty}(c_{m,j,k}x^{jm}-c_{m,j,k}^{\prime}x^{jm+1})\right]^{-1}t^{k},\qquad
\end{multline*}
where $u=4t/(1+t)^{2}$, $z=(1+t)x/(1-t)$, and
\[
c_{m,j,k}={\displaystyle 2\sum_{l=1}^{k}{l+jm-1 \choose l-1}{jm-1 \choose k-l}}\quad\text{and}\quad c_{m,j,k}^{\prime}=2\sum_{l=1}^{k}{l+jm \choose l-1}{jm \choose k-l}.
\]
\end{thm}

\begin{proof}
Part (a) follows immediately from Theorem \ref{t-gjcmipk} and Lemma
\ref{l-incR}. Setting $s=0$ in part (a), we obtain {\allowdisplaybreaks
\begin{align}
 \frac{1}{1-t}+\frac{1+t}{2(1-t)}\sum_{n=1}^{\infty}P_{12\cdots m,n}^{\ipk}(u)z^{n}
 &= \sum_{n=0}^{\infty}\Bigg(\frac{2tx}{(1-t)^{2}}-\frac{2tz^{m}}{(1-t^{2})(1+z+z^{2}+\cdots+z^{m-1})}\Bigg)^{*\left\langle n\right\rangle }\nonumber \\
 &= \sum_{n=0}^{\infty}\Bigg(\frac{2tz}{1-t^{2}}-\frac{2tz^{m}(1-z)}{(1-t^{2})(1-z^{m})}\Bigg)^{*\left\langle n\right\rangle }\label{e-monoipkmid}\\
 &= \sum_{n=0}^{\infty}\left(\frac{2tz(1-z^{m-1})}{(1-t^{2})(1-z^{m})}\right)^{*\left\langle n\right\rangle },\nonumber 
\end{align}
thus proving part (b). For part (c), we shall use the well-known identities
\[
\frac{1}{(1-t)^{n+1}}=\sum_{k=0}^{\infty}{n+k \choose k}t^{k}\qquad\text{and}\qquad(1+t)^{n}=\sum_{k=0}^{n}{n \choose k}t^{k},
\]
which imply 
\begin{align*}
\frac{2t(1+t)^{jm-1}}{(1-t)^{jm+1}} & =\sum_{k=1}^{\infty}c_{m,j,k}t^{k}\qquad\text{and}\qquad\frac{2t(1+t)^{jm}}{(1-t)^{jm+2}}=\sum_{k=1}^{\infty}c_{m,j,k}^{\prime}t^{k}.
\end{align*}
Then, continuing from (\ref{e-monoipkmid}), we have 
\begin{align*}
 & \frac{1}{1-t}+\frac{1+t}{2(1-t)}\sum_{n=1}^{\infty}P_{12\cdots m,n}^{\ipk}(u)z^{n}\\
 & \qquad\qquad\qquad=\sum_{n=0}^{\infty}\Bigg(\frac{2tx}{(1-t)^{2}}-\frac{2tz^{m}(1-z)}{(1-t^{2})(1-z^{m})}\Bigg)^{*\left\langle n\right\rangle }\\
 & \qquad\qquad\qquad=\sum_{n=0}^{\infty}\left(\frac{2tx}{(1-t)^{2}}-\frac{2t}{1-t^{2}}\sum_{j=1}^{\infty}(z^{jm}-z^{jm+1})\right)^{*\left\langle n\right\rangle }\\
 & \qquad\qquad\qquad=\sum_{n=0}^{\infty}\left(\frac{2tx}{(1-t)^{2}}-\sum_{j=1}^{\infty}\left(\frac{2t(1+t)^{jm-1}x^{jm}}{(1-t)^{jm+1}}-\frac{2t(1+t)^{jm}x^{jm+1}}{(1-t)^{jm+2}}\right)\right)^{*\left\langle n\right\rangle }\\
 & \qquad\qquad\qquad=\sum_{n=0}^{\infty}\left(\sum_{k=1}^{\infty}2kxt^{k}-\sum_{j=1}^{\infty}\sum_{k=1}^{\infty}(c_{m,j,k}x^{jm}-c_{m,j,k}^{\prime}x^{jm+1})t^{k}\right)^{*\left\langle n\right\rangle }\\
 & \qquad\qquad\qquad=\sum_{n=0}^{\infty}\left(\sum_{k=1}^{\infty}\left(2kx-\sum_{j=1}^{\infty}(c_{m,j,k}x^{jm}-c_{m,j,k}^{\prime}x^{jm+1})\right)t^{k}\right)^{*\left\langle n\right\rangle }\\
 & \qquad\qquad\qquad=\frac{1}{1-t}+\sum_{k=1}^{\infty}\sum_{n=1}^{\infty}\left(2kx-\sum_{j=1}^{\infty}(c_{m,j,k}x^{jm}-c_{m,j,k}^{\prime}x^{jm+1})\right)^{n}t^{k}\\
 & \qquad\qquad\qquad=1+\sum_{k=1}^{\infty}\sum_{n=0}^{\infty}\left(2kx-\sum_{j=1}^{\infty}(c_{m,j,k}x^{jm}-c_{m,j,k}^{\prime}x^{jm+1})\right)^{n}t^{k}\\
 & \qquad\qquad\qquad=1+\sum_{k=1}^{\infty}\left[1-2kx+\sum_{j=1}^{\infty}(c_{m,j,k}x^{jm}-c_{m,j,k}^{\prime}x^{jm+1})\right]^{-1}t^{k};
\end{align*}
this completes the proof.}
\end{proof}
In order to use Theorem \ref{t-monoipk} to compute the polynomials
$P_{12\cdots m,n}^{\ipk}(s,t)$ and $P_{12\cdots m,n}^{\ipk}(t)$,
one must ``invert'' the expression $u=4t/(1+t)^{2}$. Let us first
replace the variable $t$ with $v$, and $u$ with $t$, to obtain
$t=4v/(1+v)^{2}$. Then, solving $t=4v/(1+v)^{2}$ for $v$ yields
$v=2t^{-1}(1-\sqrt{1-t})-1$. Thus, Theorem \ref{t-monoipk} (b) is
equivalent to
\begin{alignat*}{1}
\frac{1}{1-v}+\frac{1+v}{2(1-v)}\sum_{n=1}^{\infty}P_{12\cdots m,n}^{\ipk}(t)z^{n} & =\sum_{n=0}^{\infty}\left.\left(\frac{2tz(1-z^{m-1})}{(1-t^{2})(1-z^{m})}\right)^{*\left\langle n\right\rangle }\right|_{t\mapsto v}
\end{alignat*}
where $z=(1+t)x/(1-t)$ and $v=2t^{-1}(1-\sqrt{1-t})-1$. (Note that
substitution does not commute with Hadamard product, so we cannot
simply replace $t$ with $v$ inside the Hadamard product.) With some
additional algebraic manipulations, we get the formula
\[
\sum_{n=1}^{\infty}P_{12\cdots m,n}^{\ipk}(t)x^{n}=\frac{2(1-v)}{1+v}{\displaystyle \sum_{n=0}^{\infty}}\left.\left(\frac{2tz(1-z^{m-1})}{(1-t^{2})(1-z^{m})}\right)^{*\left\langle n\right\rangle }\right|_{x\mapsto(1-t)x/(1+t),\,t\mapsto v}-\frac{2}{1+v}
\]
where $z$ and $v$ are the same as above; this formula can be used
to compute the polynomials $P_{12\cdots m,n}^{\ipk}(t)$. We can carry
out a similar process with Theorem \ref{t-monoipk} (a) and (c), as
well as with Theorems \ref{t-incilpk}, \ref{t-decilpk}, \ref{t-transipk},
and \ref{t-transilpk} appearing later in this paper.

Tables 3\textendash 4 list the first ten polynomials $P_{12\cdots m,n}^{\ipk}(t)$
for $m=3$ and $m=4$.

\renewcommand{\arraystretch}{1.2}

\begin{table}[H]

\begin{centering}
\begin{tabular}{c|ccc|c}
$n$ & $P_{123,n}^{\ipk}(t)$ &  & $n$ & $P_{123,n}^{\ipk}(t)$\tabularnewline
\cline{1-2} \cline{2-2} \cline{4-5} \cline{5-5} 
$0$ & 1 &  & $5$ & $8t+52t^{2}+10t^{3}$\tabularnewline
$1$ & $t$ &  & $6$ & $13t+200t^{2}+136t^{3}$\tabularnewline
$2$ & $2t$ &  & $7$ & $21t+714t^{2}+1170t^{3}+112t^{4}$\tabularnewline
$3$ & $3t+2t^{2}$ &  & $8$ & $34t+2468t^{2}+8180t^{3}+2676t^{4}$\tabularnewline
$4$ & $5t+12t^{2}$ &  & $9$ & $55t+8348t^{2}+50786t^{3}+37978t^{4}+2210t^{5}$\tabularnewline
\end{tabular}\vspace{5bp}
\caption{Distribution of $\protect\ipk$ over $\mathfrak{S}_{n}(123)$}
\par\end{centering}
\end{table}

\begin{table}[H]
\centering{}%
\begin{tabular}{c|ccc|c}
$n$ & $P_{1234,n}^{\ipk}(t)$ &  & $n$ & $P_{1234,n}^{\ipk}(t)$\tabularnewline
\cline{1-2} \cline{2-2} \cline{4-5} \cline{5-5} 
$0$ & 1 &  & $5$ & $13t+82t^{2}+16t^{3}$\tabularnewline
$1$ & $t$ &  & $6$ & $24t+364t^{2}+254t^{3}$\tabularnewline
$2$ & $2t$ &  & $7$ & $44t+1502t^{2}+2553t^{3}+248t^{4}$\tabularnewline
$3$ & $4t+2t^{2}$ &  & $8$ & $81t+5976t^{2}+20436t^{3}+6840t^{4}$\tabularnewline
$4$ & $7t+16t^{2}$ &  & $9$ & $149t+23286t^{2}+146636t^{3}+112192t^{4}+6638t^{5}$\tabularnewline
\end{tabular}\vspace{5bp}
\caption{Distribution of $\protect\ipk$ over $\mathfrak{S}_{n}(1234)$}
\end{table}

The linear coefficients of $P_{123,n}^{\ipk}(t)$ are Fibonacci numbers
\cite[A000045]{oeis}, and those of $P_{1234,n}^{\ipk}(t)$ are tribonacci
numbers \cite[A000073]{oeis}.\footnote{Note that \cite{oeis} uses a different indexing for these sequences.}
In fact, we can make a more general statement relating the linear
coefficients of $P_{123,n}^{\ipk}(t)$ to ``higher-order'' Fibonacci
numbers. The \textit{Fibonacci sequence of order $k$} (also called
the $k$-\textit{generalized Fibonacci sequence}) $\{f_{n}^{(k)}\}_{n\geq0}$
is defined by the recursion 
\[
f_{n}^{(k)}\coloneqq f_{n-1}^{(k)}+f_{n-2}^{(k)}+\cdots+f_{n-k}^{(k)}
\]
with $f_{0}^{(k)}\coloneqq1$ (and where we treat $f_{n}^{(k)}$ as
0 for $n<0$). Hence, the Fibonacci sequence of order two is the usual
Fibonacci sequence, and the Fibonacci sequence of order three is the
tribonacci sequence. We give three proofs of the following claim in
\cite{Zhuang2021a}.
\begin{claim}
\label{cl-ipk}Let $n\geq1$ and $m\geq3$. The number of permutations
$\pi$ in $\mathfrak{S}_{n}(12\cdots m)$ with $\ipk(\pi)=0$ is equal
to the $(m-1)$th order Fibonacci number $f_{n}^{(m-1)}$.
\end{claim}

\subsection{Monotone patterns and inverse left peak number}

Finally, we produce analogous formulas for the inverse left peak polynomials
$P_{12\cdots m,n}^{\ilpk}(s,t)$ and $P_{12\cdots m,n}^{\ilpk}(t)$.
We omit the proofs of these formulas, as they follow essentially the
same steps as the proof of Theorem \ref{t-monoipk}.
\begin{thm}
\label{t-incilpk} Let $m\geq2$. We have \leqnomode
\begin{alignat*}{1}
\tag{{a}}\frac{1}{1-t}\sum_{n=0}^{\infty}P_{12\cdots m,n}^{\ilpk}(s,u)z^{n} & =\sum_{n=0}^{\infty}\Bigg(\frac{z}{1-t}+\frac{(s-1)z^{m}}{(1-t)(1-(s-1)\sum_{l=1}^{m-1}z^{l})}\Bigg)^{*\left\langle n\right\rangle },
\end{alignat*}
\begin{align*}
\tag{{b}}\frac{1}{1-t}\sum_{n=0}^{\infty}P_{12\cdots m,n}^{\ilpk}(u)z^{n} & =\sum_{n=0}^{\infty}\left(\frac{z(1-z^{m-1})}{(1-t)(1-z^{m})}\right)^{*\left\langle n\right\rangle },
\end{align*}
and
\begin{align*}
\tag{{c}}\frac{1}{1-t}\sum_{n=0}^{\infty}P_{12\cdots m,n}^{\ilpk}(u)z^{n} & =\sum_{k=0}^{\infty}\left[\sum_{j=0}^{\infty}(d_{m,j,k}x^{jm}-d_{m,j,k}^{\prime}x^{jm+1})\right]^{-1}t^{k},
\end{align*}
where $u=4t/(1+t)^{2}$, $z=(1+t)x/(1-t)$, and
\[
d_{m,j,k}=\sum_{l=0}^{k}{l+jm \choose l}{jm \choose k-l}\qquad\text{and}\qquad d_{m,j,k}^{\prime}=\sum_{l=0}^{k}{l+jm+1 \choose l}{jm+1 \choose k-l}.
\]
\end{thm}

In Tables 5\textendash 6, we display the first ten polynomials $P_{12\cdots m,n}^{\ilpk}(t)$
for $m=3$ and $m=4$.

\renewcommand{\arraystretch}{1.2}

\begin{table}[H]
\centering{}%
\begin{tabular}{c|ccc|c}
$n$ & $P_{123,n}^{\ilpk}(t)$  &  & $n$ & $P_{123,n}^{\ilpk}(t)$\tabularnewline
\cline{1-2} \cline{2-2} \cline{4-5} \cline{5-5} 
$0$ & 1 &  & $5$ & $27t+43t^{2}$\tabularnewline
$1$ & 1 &  & $6$ & $63t+248t^{2}+38t^{3}$\tabularnewline
$2$ & $1+t$ &  & $7$ & $144t+1225t^{2}+648t^{3}$\tabularnewline
$3$ & $5t$ &  & $8$ & $333t+5591t^{2}+6882t^{3}+552t^{4}$\tabularnewline
$4$ & $12t+5t^{2}$ &  & $9$ & $765t+24304t^{2}+58552t^{3}+15756t^{4}$\tabularnewline
\end{tabular}\vspace{5bp}
\caption{Distribution of $\protect\ilpk$ over $\mathfrak{S}_{n}(123)$}
\end{table}

\begin{table}[H]
\centering{}%
\begin{tabular}{c|ccc|c}
$n$ & $P_{1234,n}^{\ilpk}(t)$ &  & $n$ & $P_{1234,n}^{\ilpk}(t)$\tabularnewline
\cline{1-2} \cline{2-2} \cline{4-5} \cline{5-5} 
$0$ & 1 &  & $5$ & $50t+61t^{2}$\tabularnewline
$1$ & 1 &  & $6$ & $138t+443t^{2}+61t^{3}$\tabularnewline
$2$ & $1+t$ &  & $7$ & $378t+2659t^{2}+1289t^{3}$\tabularnewline
$3$ & $1+5t$ &  & $8$ & $1042t+14501t^{2}+16524t^{3}+1266t^{4}$\tabularnewline
$4$ & $18t+5t^{2}$ &  & $9$ & $2866t+74941t^{2}+167780t^{3}+43314t^{4}$\tabularnewline
\end{tabular}\vspace{5bp}
\caption{Distribution of $\protect\ilpk$ over $\mathfrak{S}_{n}(1234)$}
\end{table}

Unfortunately, we cannot use symmetries to translate Theorem \ref{t-incilpk}
into a result about the pattern $m\cdots21$. However, we can obtain
an analogous result for the polynomials $P_{m\cdots21,n}^{\ilpk}(s,t)$
and $P_{m\cdots21,n}^{\ilpk}(t)$ separately; this is given below.
The proof is omitted but follows the same general line of reasoning
as in the previous few theorems. We only note that the underlying
permutation of any $m\cdots21$-cluster has exactly one left peak
(rather than having no left peaks as for $12\cdots m$-clusters),
which results in the generating function 
\[
\sum_{k=2}^{\infty}R_{m\cdots21,k}^{\ilpk}(s,t)x^{k}=\frac{stx^{m}}{1-s\sum_{l=1}^{m-1}x^{l}}
\]
for the refined cluster polynomials $R_{m\cdots21,n}^{\ilpk}(s,t)$.
\begin{thm}
\label{t-decilpk} Let $m\geq2$. We have \leqnomode
\begin{multline*}
\tag{{a}}\qquad\frac{1}{1-t}\sum_{n=0}^{\infty}P_{m\cdots21,n}^{\ilpk}(s,u)z^{n}\\
=\sum_{n=0}^{\infty}\Bigg(\frac{z}{1-t}+\frac{4t(s-1)z^{m}}{(1-t^{2})(1+t)(1-(s-1)\sum_{l=1}^{m-1}z^{l})}\Bigg)^{*\left\langle n\right\rangle },\qquad
\end{multline*}
\begin{align*}
\tag{{b}}\frac{1}{1-t}\sum_{n=0}^{\infty}P_{m\cdots21,n}^{\ilpk}(u)z^{n} & =\sum_{n=0}^{\infty}\left(\frac{(1+t)^{2}z-4tz^{m}-(1-t)^{2}z^{m+1}}{(1-t^{2})(1+t)(1-z^{m})}\right)^{*\left\langle n\right\rangle },
\end{align*}
and
\begin{multline*}
\tag{{c}}\qquad\frac{1}{1-t}\sum_{n=0}^{\infty}P_{m\cdots21,n}^{\ilpk}(u)z^{n}\\
=\frac{1}{1-x}+\sum_{k=1}^{\infty}\left[1-(2k+1)x+\sum_{j=1}^{\infty}(e_{m,j,k}x^{jm}-e_{m,j,k}^{\prime}x^{jm+1})\right]^{-1}t^{k},\qquad
\end{multline*}
where $u=4t/(1+t)^{2}$, $z=(1+t)x/(1-t)$, and
\[
e_{m,j,k}=4\sum_{l=1}^{k}{l+jm-1 \choose l-1}{jm-2 \choose k-l}\quad\text{and}\quad e_{m,j,k}^{\prime}=4\sum_{l=1}^{k}{l+jm \choose l-1}{jm-1 \choose k-l}.
\]
\end{thm}

The first ten polynomials $P_{m\cdots21,n}^{\ilpk}(t)$ for $m=3$
and $m=4$ are given in Tables 7\textendash 8.

\renewcommand{\arraystretch}{1.2}

\begin{table}[H]
\centering{}%
\begin{tabular}{c|ccc|c}
$n$ & $P_{321,n}^{\ilpk}(t)$  &  & $n$ & $P_{321,n}^{\ilpk}(t)$ \tabularnewline
\cline{1-2} \cline{2-2} \cline{4-5} \cline{5-5} 
$0$ & 1 &  & $5$ & $1+37t+32t^{2}$\tabularnewline
$1$ & 1 &  & $6$ & $1+101t+222t^{2}+25t^{3}$\tabularnewline
$2$ & $1+t$ &  & $7$ & $1+269t+1251t^{2}+496t^{3}$\tabularnewline
$3$ & $1+4t$ &  & $8$ & $1+710t+6349t^{2}+5899t^{3}+399t^{4}$\tabularnewline
$4$ & $1+13t+3t^{2}$ &  & $9$ & $1+1865t+30186t^{2}+54825t^{3}+12500t^{4}$\tabularnewline
\end{tabular}\vspace{5bp}
\caption{Distribution of $\protect\ilpk$ over $\mathfrak{S}_{n}(321)$}
\end{table}

\begin{table}[H]
\centering{}%
\begin{tabular}{c|ccc|c}
$n$ & $P_{4321,n}^{\ilpk}(t)$  &  & $n$ & $P_{4321,n}^{\ilpk}(t)$ \tabularnewline
\cline{1-2} \cline{2-2} \cline{4-5} \cline{5-5} 
$0$ & 1 &  & $5$ & $1+53t+57t^{2}$\tabularnewline
$1$ & 1 &  & $6$ & $1+158t+428t^{2}+55t^{3}$\tabularnewline
$2$ & $1+t$ &  & $7$ & $1+462t+2668t^{2}+1195t^{3}$\tabularnewline
$3$ & $1+5t$ &  & $8$ & $1+1342t+15074t^{2}+15765t^{3}+1151t^{4}$\tabularnewline
$4$ & $1+17t+5t^{2}$ &  & $9$ & $1+3886t+80338t^{2}+164337t^{3}+40339t^{4}$\tabularnewline
\end{tabular}\vspace{5bp}
\caption{Distribution of $\protect\ilpk$ over $\mathfrak{S}_{n}(4321)$}
\end{table}

The linear coefficients of the $P_{321,n}^{\ilpk}(t)$ match OEIS
sequence A080145 \cite[A080145]{oeis}, which involves the Fibonacci
numbers $f_{n}\coloneqq f_{n}^{(2)}$. We give two proofs of this
claim in \cite{Zhuang2021a}.
\begin{claim}
\label{cl-ilpk}Let $n\geq1$. The number of permutations $\pi$ in
$\mathfrak{S}_{n}(321)$ with $\ilpk(\pi)=1$ is equal to 
\[
\sum_{i=1}^{n-1}\sum_{j=1}^{i}f_{j-1}f_{j}=f_{n-1}f_{n}-\left\lfloor \frac{n+1}{2}\right\rfloor .
\]
\end{claim}

\section{Transpositional patterns \texorpdfstring{$12\cdots(a-1)(a+1)a(a+2)(a+3)\cdots m$}{12...(a-1)(a+1)a(a+2)(a+3)...m}}

\subsection{Cluster generating functions for transpositional patterns}

In this section, we turn our attention to patterns of the form $\sigma=12\cdots(a-1)(a+1)a(a+2)(a+3)\cdots m$
where $m\geq5$ and $2\leq a\leq m-2$. These are precisely the elementary
transpositions $(a,a+1)$ of $\mathfrak{S}_{m}$\textemdash aside
from the transpositions $(1,2)$ and $(m-1,m)$\textemdash and form
another family of patterns for which it is straightforward to obtain
closed-form generating functions for our refined cluster polynomials.
\begin{lem}
\label{l-etR} Let $\sigma=12\cdots(a-1)(a+1)a(a+2)(a+3)\cdots m$
where $m\geq5$ and $2\leq a\leq m-2$. Let $i=\min(a,m-a)$. Then
\begin{alignat*}{1}
\sum_{k=2}^{\infty}R_{\sigma,k}^{\ides}(s,t)x^{k} & =\sum_{k=2}^{\infty}R_{\sigma,k}^{\ipk}(s,t)x^{k}=\frac{st^{2}x^{m}}{1-st\sum_{l=1}^{i}x^{m-l}}
\end{alignat*}
and
\[
\sum_{k=2}^{\infty}R_{\sigma,k}^{\ilpk}(s,t)x^{k}=\frac{stx^{m}}{1-st\sum_{l=1}^{i}x^{m-l}}.
\]
\end{lem}

\begin{proof}
Let us first assume that $m-a\leq a$; then the overlap set of $\sigma$
is given by $O_{\sigma}=\{a,a+1,\dots,m-1\}$. In this case, we can
uniquely generate $\sigma$-clusters by first taking the permutation
$\sigma$ and then repeatedly appending the next $l$ largest integers
(where $a\leq l\leq m-1$) in the order compatible with the pattern
$\sigma$\textemdash each iteration creates an additional marked occurrence
of $\sigma$. See Figure 2 for an illustration in the case $m=5$
and $a=3$. Thus, we have the formula 
\begin{align*}
\sum_{k=2}^{\infty}\sum_{\pi\in\mathfrak{S}_{k}}\sum_{c\in C_{\sigma,\pi}}s^{\mk_{\sigma}(c)}x^{k} & =\frac{sx^{m}}{1-s(x^{a}+x^{a+1}+\cdots+x^{m-1})}
\end{align*}
when $m-a\leq a$ (see also \cite[p. 357]{Elizalde2012}). 
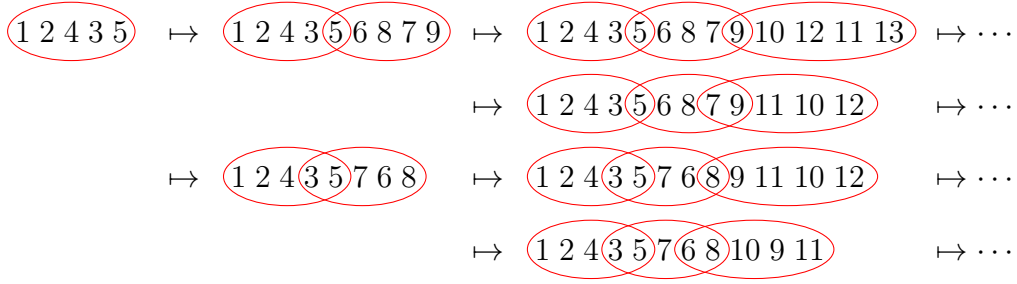
\begin{figure}
\begin{center}
\begin{tikzpicture}

\node at (0,0){
\begin{minipage}{\textwidth}
\begin{alignat*}{3} 
1\;2\;4\;3\;5\quad\mapsto\quad & 1\;2\;4\;3\;5\;6\;8\;7\;9 & \quad\mapsto\quad & 1\;2\;4\;3\;5\;6\;8\;7\;9\;10\;12\;11\;13 & \quad\mapsto\cdots\\[10pt]  &  & \quad\mapsto\quad & 1\;2\;4\;3\;5\;6\;8\;7\;9\;11\;10\;12 & \quad\mapsto\cdots\\[10pt] \mapsto\quad & 1\;2\;4\;3\;5\;7\;6\;8 & \quad\mapsto\quad & 1\;2\;4\;3\;5\;7\;6\;8\;9\;11\;10\;12 & \quad\mapsto\cdots\\[10pt]  &  & \quad\mapsto\quad & 1\;2\;4\;3\;5\;7\;6\;8\;10\;9\;11 & \quad\mapsto\cdots 
\end{alignat*}
\end{minipage}
};


\draw[red] (-5.89,1.2) ellipse (24bp and 11bp);
\draw[red] (-3.02,1.2) ellipse (24bp and 11bp);
\draw[red] (-1.7,1.2) ellipse (24bp and 11bp);
\draw[red] (1.02,1.2) ellipse (24bp and 11bp);
\draw[red] (2.32,1.2) ellipse (24bp and 11bp);
\draw[red] (4.05,1.2) ellipse (36.5bp and 11bp);

\draw[red] (1.02,0.24) ellipse (24bp and 11bp);
\draw[red] (2.32,0.24) ellipse (24bp and 11bp);
\draw[red] (3.63,0.24) ellipse (34bp and 11bp);

\draw[red] (-3.02,-0.73) ellipse (24bp and 11bp);
\draw[red] (-2.02,-0.73) ellipse (24bp and 11bp);
\draw[red] (1.02,-0.73) ellipse (24bp and 11bp);
\draw[red] (2.01,-0.73) ellipse (24bp and 11bp);
\draw[red] (3.63,-0.73) ellipse (34bp and 11bp);

\draw[red] (1.02,-1.7) ellipse (24bp and 11bp);
\draw[red] (2.01,-1.7) ellipse (24bp and 11bp);
\draw[red] (3.19,-1.7) ellipse (30bp and 11bp);

\end{tikzpicture}
\end{center}

\caption{$12435$-clusters}
\end{figure}

When $a\leq m-a$, we instead have $O_{\sigma}=\{m-a,m-a+1,\dots,m-1\}$,
and we can uniquely generate clusters in essentially the same way
as above; the only differences are that (1) $m-a\leq l\leq m-1$,
and (2) whenever we append $l=m-a$ integers to create a larger cluster,
we increase the last entry of the existing cluster by 1 so that the
new cluster contains an additional occurrence of $\sigma$. See Figure
3 for an illustration in the case $m=5$ and $a=2$. Thus, we have
the formula 
\begin{align*}
\sum_{k=2}^{\infty}\sum_{\pi\in\mathfrak{S}_{k}}\sum_{c\in C_{\sigma,\pi}}s^{\mk_{\sigma}(c)}x^{k} & =\frac{sx^{m}}{1-s(x^{m-a}+x^{m-a+1}+\cdots+x^{m-1})}
\end{align*}
when $a\leq m-a$. In either case, we have
\[
\sum_{k=2}^{\infty}\sum_{\pi\in\mathfrak{S}_{k}}\sum_{c\in C_{\sigma,\pi}}s^{\mk_{\sigma}(c)}x^{k}=\frac{sx^{m}}{1-s(x^{m-i}+x^{m-i+1}+\cdots+x^{m-1})}=\frac{sx^{m}}{1-s\sum_{l=1}^{i}x^{m-l}}
\]
where $i=\min(a,m-a)$.
\begin{figure}
\begin{center}
\begin{tikzpicture}

\node at (0,0){
\begin{minipage}{\textwidth}
\begin{alignat*}{3} 
1\;3\;2\;4\;5\quad\mapsto\quad & 1\;3\;2\;4\;5\;7\;6\;8\;9 & \quad\mapsto\quad & 1\;3\;2\;4\;5\;7\;6\;8\;9\;11\;10\;12\;13 & \quad\mapsto\cdots\\[10pt]  &  & \quad\mapsto\quad & 1\;3\;2\;4\;5\;7\;6\;8\;10\;9\;11\;12 & \quad\mapsto\cdots\\[10pt] \mapsto\quad & 1\;3\;2\;4\;6\;5\;7\;8 & \quad\mapsto\quad & 1\;3\;2\;4\;6\;5\;7\;8\;10\;9\;11\;12 & \quad\mapsto\cdots\\[10pt]  &  & \quad\mapsto\quad & 1\;3\;2\;4\;6\;5\;7\;9\;8\;10\;11 & \quad\mapsto\cdots 
\end{alignat*}
\end{minipage}
};


\draw[red] (-5.89,1.2) ellipse (24bp and 11bp);
\draw[red] (-3.02,1.2) ellipse (24bp and 11bp);
\draw[red] (-1.7,1.2) ellipse (24bp and 11bp);
\draw[red] (1.02,1.2) ellipse (24bp and 11bp);
\draw[red] (2.32,1.2) ellipse (24bp and 11bp);
\draw[red] (4.05,1.2) ellipse (36.5bp and 11bp);

\draw[red] (1.02,0.24) ellipse (24bp and 11bp);
\draw[red] (2.42,0.24) ellipse (27bp and 11bp);
\draw[red] (3.63,0.24) ellipse (34bp and 11bp);

\draw[red] (-3.02,-0.73) ellipse (24bp and 11bp);
\draw[red] (-2.02,-0.73) ellipse (24bp and 11bp);
\draw[red] (1.02,-0.73) ellipse (24bp and 11bp);
\draw[red] (2.01,-0.73) ellipse (24bp and 11bp);
\draw[red] (3.63,-0.73) ellipse (34bp and 11bp);

\draw[red] (1.02,-1.7) ellipse (24bp and 11bp);
\draw[red] (2.01,-1.7) ellipse (24bp and 11bp);
\draw[red] (3.16,-1.7) ellipse (30bp and 11bp);

\end{tikzpicture}
\end{center}

\caption{$13245$-clusters}
\end{figure}
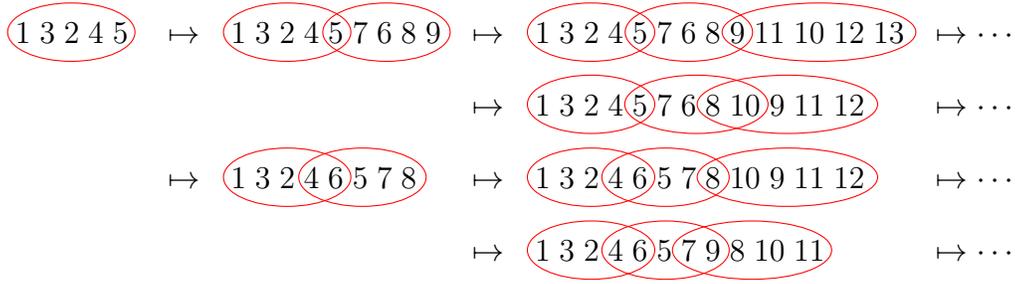

It is clear that the underlying permutation of any $\sigma$-cluster
is a product of disjoint elementary transpositions, hence an involution.
So, whenever $\pi$ is the underlying permutation of a $\sigma$-cluster,
we have $\des(\pi)=\ides(\pi)$, $\pk(\pi)=\ipk(\pi)$, and $\lpk(\pi)=\ilpk(\pi)$.
Each marked occurrence of a $\sigma$-cluster $c$ contributes exactly
one descent, which is also a peak and a left peak. Hence, we have
{\allowdisplaybreaks 
\begin{align*}
\sum_{k=2}^{\infty}R_{\sigma,k}^{\ides}(s,t)x^{k} & =\sum_{k=2}^{\infty}\sum_{\pi\in\mathfrak{S}_{k}}t^{\ides(\pi)+1}\sum_{c\in C_{12\cdots m,\pi}}s^{\mk_{12\cdots m}(c)}x^{k}\\
 & =t\sum_{k=2}^{\infty}\sum_{\pi\in\mathfrak{S}_{k}}\sum_{c\in C_{12\cdots m,\pi}}(st)^{\mk_{12\cdots m}(c)}x^{k}\\
 & =\frac{st^{2}x^{m}}{1-st\sum_{l=1}^{i}x^{m-l}}.
\end{align*}
Our formulas for $\sum_{k=2}^{\infty}R_{\sigma,k}^{\ipk}(s,t)x^{k}$
and $\sum_{k=2}^{\infty}R_{\sigma,k}^{\ilpk}(s,t)x^{k}$ are obtained
in the same way.}
\end{proof}

\subsection{Transpositional patterns and inverse descent number}

We now apply our results from Section 3.5 to the patterns $\sigma=12\cdots(a-1)(a+1)a(a+2)(a+3)\cdots m$
for arbitrary $m\geq5$ and $2\leq a\leq m-2$. All of these formulas
follow immediately from combining either Theorem \ref{t-gjcmides},
\ref{t-gjcmipk}, or \ref{t-gjcmilpk} with Lemma \ref{l-etR}, and
then setting $s=0$.
\begin{thm}
\label{t-transides} Let $\sigma=12\cdots(a-1)(a+1)a(a+2)(a+3)\cdots m$
where $m\geq5$ and $2\leq a\leq m-2$. Let $i=\min(a,m-a)$. We have
\leqnomode
\begin{alignat*}{1}
\tag{{a}}\sum_{n=0}^{\infty}\frac{A_{\sigma,n}^{\ides}(s,t)}{(1-t)^{n+1}}x^{n} & =\sum_{n=0}^{\infty}\Bigg(\frac{tx}{(1-t)^{2}}+\frac{(s-1)t^{2}z^{m}}{(1-t)(1-(s-1)t\sum_{l=1}^{i}z^{m-l})}\Bigg)^{*\left\langle n\right\rangle }
\end{alignat*}
and
\begin{alignat*}{1}
\tag{{b}}\sum_{n=0}^{\infty}\frac{A_{\sigma,n}^{\ides}(t)}{(1-t)^{n+1}}x^{n} & =\sum_{n=0}^{\infty}\Bigg(\frac{tx}{(1-t)^{2}}-\frac{t^{2}z^{m}}{(1-t)(1+t\sum_{l=1}^{i}z^{m-l})}\Bigg)^{*\left\langle n\right\rangle }
\end{alignat*}
where $z=x/(1-t)$.
\end{thm}

We use Theorem \ref{t-transides} to compute the first ten polynomials
$A_{13245,n}^{\ides}(t)$; see Table 9.

\renewcommand{\arraystretch}{1.2}

\begin{table}[H]
\centering{}%
\begin{tabular}{c|c}
$n$ & $A_{13245,n}^{\ides}(t)$\tabularnewline
\hline 
$0$ & 1\tabularnewline
$1$ & $t$\tabularnewline
$2$ & $t+t^{2}$\tabularnewline
$3$ & $t+4t^{2}+t^{3}$\tabularnewline
$4$ & $t+11t^{2}+11t^{3}+t^{4}$\tabularnewline
$5$ & $t+25t^{2}+66t^{3}+26t^{4}+t^{5}$\tabularnewline
$6$ & $t+53t^{2}+294t^{3}+302t^{4}+57t^{5}+t^{6}$\tabularnewline
$7$ & $t+108t^{2}+1125t^{3}+2368t^{4}+1191t^{5}+120t^{6}+t^{7}$\tabularnewline
$8$ & $t+215t^{2}+3934t^{3}+14923t^{4}+15363t^{5}+4293t^{6}+247t^{7}+t^{8}$\tabularnewline
$9$ & $t+422t^{2}+12985t^{3}+82066t^{4}+150240t^{5}+86954t^{6}+14608t^{7}+502t^{8}+t^{9}$\tabularnewline
\end{tabular}\vspace{5bp}
\caption{Distribution of $\protect\ides$ over $\mathfrak{S}_{n}(13245)$}
\end{table}
By the symmetry present in Theorem \ref{t-transides}, the polynomials
$A_{13245,n}^{\ides}(t)$ displayed above are the same as the polynomials
$A_{12435,n}^{\ides}(t)$ for corresponding $n$.

\subsection{Transpositional patterns and inverse peak number}
\begin{thm}
\label{t-transipk}Let $\sigma=12\cdots(a-1)(a+1)a(a+2)(a+3)\cdots m$
where $m\geq5$ and $2\leq a\leq m-2$. Let $i=\min(a,m-a)$. Then
\leqnomode
\begin{multline*}
\tag{{a}}\qquad\frac{1}{1-t}+\frac{1+t}{2(1-t)}\sum_{n=1}^{\infty}P_{\sigma,n}^{\ipk}(s,u)z^{n}\\
=\sum_{n=0}^{\infty}\Bigg(\frac{2tx}{(1-t)^{2}}+\frac{(1+t)(s-1)u^{2}z^{m}}{2(1-t)(1-(s-1)u\sum_{l=1}^{i}z^{m-l})}\Bigg)^{*\left\langle n\right\rangle }\qquad
\end{multline*}
and
\begin{align*}
\tag{{b}}\qquad\frac{1}{1-t}+\frac{1+t}{2(1-t)}\sum_{n=1}^{\infty}P_{\sigma,n}^{\ipk}(u)z^{n}
=\sum_{n=0}^{\infty}\Bigg(\frac{2tx}{(1-t)^{2}}-\frac{(1+t)u^{2}z^{m}}{2(1-t)(1+u\sum_{l=1}^{i}z^{m-l})}\Bigg)^{*\left\langle n\right\rangle }
\end{align*}
where $u=4t/(1+t)^{2}$ and $z=(1+t)x/(1-t)$.
\end{thm}

We give the first ten polynomials $P_{13245,n}^{\ipk}(t)$\textemdash which
are also the first ten polynomials $P_{12435,n}^{\ipk}(t)$\textemdash in
Table 10.

\renewcommand{\arraystretch}{1.2}

\begin{table}[H]
\centering{}%
\begin{tabular}{c|ccc|c}
$n$ & $P_{13245,n}^{\ipk}(t)$ &  & $n$ & $P_{13245,n}^{\ipk}(t)$\tabularnewline
\cline{1-2} \cline{2-2} \cline{4-5} \cline{5-5} 
$0$ & 1 &  & $5$ & $16t+87t^{2}+16t^{3}$\tabularnewline
$1$ & $t$ &  & $6$ & $32t+408t^{2}+268t^{3}$\tabularnewline
$2$ & $2t$ &  & $7$ & $64t+1776t^{2}+2808t^{3}+266t^{4}$\tabularnewline
$3$ & $4t+2t^{2}$ &  & $8$ & $128t+7424t^{2}+23745t^{3}+7680t^{4}$\tabularnewline
$4$ & $8t+16t^{2}$ &  & $9$ & $256t+30336t^{2}+178029t^{3}+131542t^{4}+7616t^{5}$\tabularnewline
\end{tabular}\vspace{5bp}
\caption{Distribution of $\protect\ipk$ over $\mathfrak{S}_{n}(13245)$}
\end{table}

\subsection{Transpositional patterns and inverse left peak number}
\begin{thm}
\label{t-transilpk}Let $\sigma=12\cdots(a-1)(a+1)a(a+2)(a+3)\cdots m$
where $m\geq5$ and $2\leq a\leq m-2$. Let $i=\min(a,m-a)$. Then
\leqnomode
\begin{align*}
\tag{{a}}\frac{1}{1-t}\sum_{n=0}^{\infty}P_{\sigma,n}^{\ilpk}(s,u)z^{n} & =\sum_{n=0}^{\infty}\left(\frac{z}{1-t}+\frac{(s-1)uz^{m}}{(1-t)(1-(s-1)u\sum_{l=1}^{i}z^{m-l})}\right)^{*\left\langle n\right\rangle }
\end{align*}
and
\begin{align*}
\tag{{b}}\frac{1}{1-t}\sum_{n=0}^{\infty}P_{\sigma,n}^{\ilpk}(u)z^{n} & =\sum_{n=0}^{\infty}\left(\frac{z}{1-t}-\frac{uz^{m}}{(1-t)(1+u\sum_{l=1}^{i}z^{m-l})}\right)^{*\left\langle n\right\rangle }
\end{align*}
where $u=4t/(1+t)^{2}$ and $z=(1+t)x/(1-t)$.
\end{thm}

Table 11 lists the first ten polynomials $P_{13245,n}^{\ilpk}(t)=P_{12435,n}^{\ilpk}(t)$.

\renewcommand{\arraystretch}{1.2}

\begin{table}[H]
\centering{}%
\begin{tabular}{c|ccc|c}
$n$ & $P_{13245,n}^{\ilpk}(t)$ &  & $n$ & $P_{13245,n}^{\ilpk}(t)$\tabularnewline
\cline{1-2} \cline{2-2} \cline{4-5} \cline{5-5} 
$0$ & 1 &  & $5$ & $1+57t+61t^{2}$\tabularnewline
$1$ & $1$ &  & $6$ & $1+173t+473t^{2}+61t^{3}$\tabularnewline
$2$ & $1+t$ &  & $7$ & $1+516t+3030t^{2}+1367t^{3}$\tabularnewline
$3$ & $1+5t$ &  & $8$ & $1+1528t+17551t^{2}+18536t^{3}+1361t^{4}$\tabularnewline
$4$ & $1+18t+5t^{2}$ &  & $9$ & $1+4511t+95867t^{2}+198379t^{3}+49021t^{4}$\tabularnewline
\end{tabular}\vspace{5bp}
\caption{Distribution of $\protect\ilpk$ over $\mathfrak{S}_{n}(13245)$}
\end{table}

\section{Conclusion}

In summary, we have proven a lifting of Elizalde and Noy's adaptation
of the Goulden\textendash Jackson cluster method for permutations
to the Malvenuto\textendash Reutenauer algebra $\mathbf{FQSym}$.
By applying two homomorphisms to the cluster method in $\mathbf{FQSym}$,
we recover both Elizalde and Noy's cluster method and Elizalde's $q$-cluster
method as special cases. We have also defined several other homomorphisms,
by way of the theory of shuffle-compatibility, which lead to new specializations
of our generalized cluster method that keep track of various inverse
statistics. Finally, we applied these results to two families of patterns:
the monotone patterns $12\cdots m$ and $m\cdots21$, and the transpositional
patterns $12\cdots(a-1)(a+1)a(a+2)(a+3)\cdots m$ where $m\geq5$
and $2\leq a\leq m-2$.

We chose to study monotone patterns as well as the transpositional
patterns of the form above because, for these patterns, it is easy
to count clusters by the inverse statistics that we consider. In particular,
these patterns have two nice properties:
\begin{enumerate}
\item These patterns are \textit{chain patterns}. Elizalde and Noy \cite{Elizalde2012}
showed that counting clusters is equivalent to counting linear extensions
in a certain poset, and the poset associated with a chain pattern
is a chain. This means that if we fix the length of a cluster as well
as the positions of the marked occurrences within the cluster, then
there is at most one cluster of that length and with that set of positions.
\item Clusters formed from any one of these patterns are involutions, so
counting clusters by $\ist$ is the same as counting them by $\st$.
\end{enumerate}
In forthcoming work, joint with Sergi Elizalde and Justin Troyka,
we study the transpositional patterns $2134\cdots m$ and $12\cdots(m-2)m(m-1)$
for $m\geq3$. Interestingly, the enumeration of $2134\cdots m$-clusters
and $12\cdots(m-2)m(m-1)$-clusters by $\ides$, $\ipk$, and $\ilpk$
turns out to have connections to generalized Stirling permutations
\cite{Gessel2020a} and $1/k$-Eulerian polynomials \cite{Savage2012}.
Although these are not chain patterns and their clusters are not involutions,
they are examples of \textit{non-overlapping patterns}: patterns whose
overlap set is equal to $\{m-1\}$. Both the non-overlapping condition
and the condition of being a chain pattern greatly restrict how clusters
can be formed, making them easier to characterize and thus more amenable
to study. As such, one direction of future work is to apply our results
to other families of non-overlapping patterns and chain patterns.

We also present the following conjecture, which is suggested by computational
evidence.
\begin{conjecture}
\label{cj-realroots} Let $\sigma$ be $12\cdots m$ or $m\cdots21$
where $m\geq3$, or $12\cdots(a-1)(a+1)a(a+2)(a+3)\cdots m$ where
$m\geq5$ and $2\leq a\leq m-2$. Then the polynomials $A_{\sigma,n}^{\ides}(t)$,
$P_{\sigma,n}^{\ipk}(t)$, and $P_{\sigma,n}^{\ilpk}(t)$ have only
real roots for all $n\geq2$.
\end{conjecture}

In particular, Conjecture \ref{cj-realroots} would imply that\textemdash for
all patterns $\sigma$ considered in this paper\textemdash the polynomials
$A_{\sigma,n}^{\ides}(t)$, $P_{\sigma,n}^{\ipk}(t)$, and $P_{\sigma,n}^{\ilpk}(t)$
are unimodal and log-concave, and that the distributions of the statistics
$\ides$, $\ipk$, and $\ilpk$ over $\mathfrak{S}_{n}(\sigma)$ converge
to a normal distribution as $n\rightarrow\infty$. It is worth noting
that the Eulerian, peak, and left peak polynomials {\allowdisplaybreaks
\begin{align*}
A_{n}(t) & \coloneqq\sum_{\pi\in\mathfrak{S}_{n}}t^{\des(\pi)+1}=\sum_{\pi\in\mathfrak{S}_{n}}t^{\ides(\pi)+1},\\
P_{n}^{\pk}(t) & \coloneqq\sum_{\pi\in\mathfrak{S}_{n}}t^{\pk(\pi)+1}=\sum_{\pi\in\mathfrak{S}_{n}}t^{\ipk(\pi)+1},\text{ and}\\
P_{n}^{\lpk}(t) & \coloneqq\sum_{\pi\in\mathfrak{S}_{n}}t^{\lpk(\pi)}=\sum_{\pi\in\mathfrak{S}_{n}}t^{\ilpk(\pi)}
\end{align*}
}are all real-rooted (see, e.g., \cite{Petersen2007,Petersen2015,Stembridge1997}).
In light of this fact, one might intuitively expect the polynomials
$A_{\sigma,n}^{\ides}(t)$, $P_{\sigma,n}^{\ipk}(t)$, and $P_{\sigma,n}^{\ilpk}(t)$
to be real-rooted as well, since avoiding a single consecutive pattern
is not a very restrictive condition (especially when compared to classical
pattern avoidance) and therefore might be expected to preserve unimodality
or asymptotic normality.

One can use the theory of shuffle-compatibility from \cite{Gessel2018}
to define other homomorphisms on $\mathbf{FQSym}$ which can be used
to count permutations by inverses of shuffle-compatible statistics
other than the ones we consider here. For example, the bistatistic
$(\pk,\des)$ is shuffle-compatible, so we can define a homomorphism
$\Psi_{(\ipk,\ides)}$ that can be used to produce yet another specialization
of our generalized cluster method that simultaneously refines by $\ipk$
and $\ides$. Finally, in a different direction, one may apply our
homomorphisms to other formulas in $\mathbf{FQSym}$ which lift classical
formulas in permutation enumeration\textemdash such as the lifting
of Andr\'e's exponential generating function considered in \cite{Josuat-Verges2012}\textemdash leading
to new refinements of these classical formulas by inverse statistics.

\bigskip{}

\noindent \textbf{Acknowledgements.} The author thanks Justin Troyka
for carefully reading a preliminary version of this paper and offering
numerous suggestions and corrections, two anonymous referees for providing
additional suggestions and corrections, and Sergi Elizalde for helpful
discussions relating to this work. The author is also grateful to
Ira Gessel, from whom he first learned about the Malvenuto\textendash Reutenauer
algebra and the Goulden\textendash Jackson cluster method. The author
was partially supported by an AMS-Simons Travel Grant.

\bibliographystyle{plain}
\bibliography{bibliography}

\end{document}